\documentclass[twoside]{article}%
\usepackage{mathtools}
\usepackage{latexsym}
\usepackage{mathrsfs}
  \usepackage{hyperref}
\usepackage{pstricks}
\usepackage{amssymb,bm}
\usepackage{amsmath}
\usepackage{amsfonts}
\usepackage{amsthm}
\usepackage{tikz}
\usepackage{epsfig}
\usepackage{verbatim}
\usepackage{float}
\usepackage{tabularx}
\usepackage{listings}
\providecommand{\U}[1]{\protect\rule{.1in}{.1in}}

  \newcolumntype{Y}{>{\raggedleft\arraybackslash}X}

\def\vs{\vskip.3cm}
\def\noi{\noindent}

\def\gdeg{G\text{\rm -deg}}

\usepackage{listings}
  \lstdefinelanguage{GAP}{
    basicstyle=\ttfamily,
    keywords={true, false, function, return, fail, if, in, while, do, od, else, elif, fi, break, continue},
    keywordstyle=\color{blue}\bfseries,
    otherkeywords={
      >, <, ==
    },
    identifierstyle=\color{black},
    sensitive=True,
    comment=[l]{\#},
    commentstyle=\color{cyan},
    stringstyle=\color{red},
    morestring=[b]',
    morestring=[b]"
  }

\def\ve{\varepsilon}

\DeclareMathOperator{\id}{Id}

\newrgbcolor{violet}{.6 .1 .8}
  \newrgbcolor{lightyellow}{1 1 .8}
  \newrgbcolor{lightblue}{.80 1 1}
  \newrgbcolor{mygreen}{0 .66 .05}
  \definecolor{mygreen}{rgb}{0,.66,.05}
  \definecolor{lightyellow}{rgb}{1,1,.80}
  \newrgbcolor{orange}{1 .6 0}
  \newrgbcolor{GreenYellow}{.85 1 .31}
  \newrgbcolor{Yellow}{1  1  0}
  \newrgbcolor{Goldenrod}{1  .90  .16}
  \newrgbcolor{Dandelion}{1  .71  .16}
  \newrgbcolor{Apricot}{1  .68  .48}
  \newrgbcolor{Peach}{1  .50  .30}
  \newrgbcolor{Melon}{1  .54  .50}
  \newrgbcolor{YellowOrange}{1  .58  0}
  \newrgbcolor{Orange}{1  .39  .13}
  \newrgbcolor{BurntOrange}{1  .49  0}
  \newrgbcolor{Bittersweet}{1.  .4300  .24}
  \newrgbcolor{RedOrange}{1  .23  .13}
  \newrgbcolor{Mahogany}{1.  .4475  .4345}
  \newrgbcolor{Maroon}{1.  .4084  .5376}
  \newrgbcolor{BrickRed}{1.  .3592  .3232}
  \newrgbcolor{Red}{1  0  0}
  \newrgbcolor{OrangeRed}{1  0  .50}
  \newrgbcolor{RubineRed}{1  0  .87}
  \newrgbcolor{WildStrawberry}{1  .04  .61}
  \newrgbcolor{CarnationPink}{1  .37  1}
  \newrgbcolor{Salmon}{1  .47  .62}
  \newrgbcolor{Magenta}{1  0  1}
  \newrgbcolor{VioletRed}{1  .19  1}
  \newrgbcolor{Rhodamine}{1  .18  1}
  \newrgbcolor{Mulberry}{.6668  .1180  1.}
  \newrgbcolor{RedViolet}{.9538  .4060  1.}
  \newrgbcolor{Fuchsia}{.5676  .1628  1.}
  \newrgbcolor{Lavender}{1  .52  1}
  \newrgbcolor{Thistle}{.88  .41  1}
  \newrgbcolor{Orchid}{.68  .36  1}
  \newrgbcolor{DarkOrchid}{.60  .20  .80}
  \newrgbcolor{Purple}{.55  .14  1}
  \newrgbcolor{Plum}{.50  0  1}
  \newrgbcolor{Violet}{.98 .15 .95}
  \newrgbcolor{RoyalPurple}{.25  .10  1}
  \newrgbcolor{BlueViolet}{.84  .38  .98}
  \newrgbcolor{Periwinkle}{.43  .45  1}
  \newrgbcolor{CadetBlue}{.38  .43  .77}
  \newrgbcolor{CornflowerBlue}{.35  .87  1}
  \newrgbcolor{MidnightBlue}{.4414  .9259  1.}
  \newrgbcolor{NavyBlue}{.06  .46  1}
  \newrgbcolor{RoyalBlue}{0  .50  1}
  \newrgbcolor{Blue}{0  0  1}
  \newrgbcolor{Cerulean}{.06  .89  1}
  \newrgbcolor{Cyan}{0  1  1}
  \newrgbcolor{ProcessBlue}{.04  1  1}
  \newrgbcolor{SkyBlue}{.38  1  .88}
  \newrgbcolor{Turquoise}{.15  1  .80}
  \newrgbcolor{TealBlue}{.1572  1.  .6668}
  \newrgbcolor{Aquamarine}{.18  1  .70}
  \newrgbcolor{BlueGreen}{.15  1  .67}
  \newrgbcolor{Emerald}{0  1  .50}
  \newrgbcolor{JungleGreen}{.01  1  .48}
  \newrgbcolor{SeaGreen}{.31  1  .50}
  \newrgbcolor{Green}{0  1  0}
  \newrgbcolor{ForestGreen}{.1992  1.  .2256}
  \newrgbcolor{PineGreen}{.3100  1.  .5575}
  \newrgbcolor{LimeGreen}{.50  1  0}
  \newrgbcolor{YellowGreen}{.56  1  .26}
  \newrgbcolor{SpringGreen}{.74  1  .24}
  \newrgbcolor{OliveGreen}{.6160  1.  .4300}
  \newrgbcolor{RawSienna}{.53  .28  .16}
  \newrgbcolor{Sepia}{1.  .7510  .70}
  \newrgbcolor{Brown}{.41  .25  .18}
  \newrgbcolor{TAN}{.86  .58  .44}
  \newrgbcolor{Gray}{1.  1.  1.}
  \newrgbcolor{Black}{1  1  1}
  \newrgbcolor{White}{1  1  1}

\newtheorem{theorem}{Theorem}[section]
\newtheorem{proposition}[theorem]{Proposition}
\newtheorem{lemma}[theorem]{Lemma}
\newtheorem{corollary}[theorem]{Corollary}

{\bfseries}{\rmfamily}
\newtheorem{definition}[theorem]{Definition}

\newtheorem{remark-definition}[theorem]{Remark and Definition}
\begin{document}

\title{Subharmonic Solutions in Reversible Non-Autonomous Differential Equations}
\author{Izuchukwu Eze\thanks{{\small Department of Mathematical Sciences the
University of Texas at Dallas Richardson, 75080 USA, amos.eze@utdallas.edu}},
Carlos Garc\'{\i}a-Azpeitia\thanks{{\small Depto. Matem\'{a}ticas y
Mec\'{a}nica IIMAS, Universidad Nacional Aut\'{o}noma de M\'{e}xico, Apdo.
Postal 20-726, 01000 Ciudad de M\'{e}xico, M\'{e}xico, and Department of Mathematics, Xiangnan University, 889 Chen Zhou Da Dao, Chenzhou, Hunan 423000, China.
cgazpe@mym.iimas.unam.mx}}, Wieslaw Krawcewicz\thanks{{\small Applied
Mathematics Center at Guangzhou University, Guangzhou, 510006 China, and
Department of Mathematical Sciences the University of Texas at Dallas
Richardson, 75080 USA. wieslaw@utallas.edu}} and Yanli
Lv\thanks{{\small Department of Mathematical Sciences the China Three Gorges
University, Yichang, China , rebecca\_utd@aliyun.com}}}
\date{}
\maketitle

\abstract{We study the existence of subharmonic solutions in the system $\ddot
{u}(t)=f(t,u(t))$, where $u(t)\in\mathbb{R}^{k}$ and $f$ is an even and
$p$-periodic function in time. Under some additional symmetry conditions on
 the function $f$, the problem of finding $mp$-periodic
solutions can be reformulated in a functional space as a $\Gamma\times\mathbb{Z}_{2}\times D_{m}%
$-equivariant equation, where the group $\Gamma\times\mathbb{Z}_{2}$ acts on the  space $\mathbb{R}^{k}$ and
$D_{m}$ acts on $u(t)$ by  time-shifts and reflection. We apply Brouwer equivariant 
degree to prove the existence of an infinite number of subharmonic solutions
for the function $f$  that satisfies additional hypothesis on linear behavior near zero and the
Nagumo condition at infinity. We also discuss the bifurcation of subharmonic
solutions when the system depends on an extra parameter.
}

\section{Introduction}

In this paper we study the existence of subharmonic solutions of the system
\begin{equation}
\ddot{u}(t)=f(t,u(t)),\;\;u(t)\in\mathbb{R}^{k},\label{eq:sys-non-aut}%
\end{equation}
where $f:\mathbb{R}\times\mathbb{R}^{k}\rightarrow\mathbb{R}^{k}$ is a
continuous function satisfying the following conditions:

\begin{itemize} \itemindent=2pt\labelsep=3pt\labelwidth5pt\itemsep=1pt

\item[($A_{1}$)] ~~ For all $t\in\mathbb{R}$ and $x\in\mathbb{R}^{k}$ we have
$f(t+2\pi,x)=f(t,x)$;

\item[($A_{2}$)] ~~ For all $t\in\mathbb{R}$ and $x\in\mathbb{R}^{k}$ we have
$f(-t,x)=f(t,x)$;

\item[($A_{3}$)] ~~ For all $t\in\mathbb{R}$ and $x\in\mathbb{R}^{k}$ we have
$f(t,-x)=-f(t,x)$.
\end{itemize}

Notice that in condition ($A_{1}$), one could consider $f(t,x)$ being $p$-periodic with
respect to $t$, however by rescaling the time $t$, one can always arrive to a
$2\pi$-periodic function. 
\vs
The problem of finding multiple subharmonic
solutions to \eqref{eq:sys-non-aut}, especially in the case of Hamiltonian
systems of the type
\begin{equation}
\ddot{u}+\nabla F(t,u)=h(t)\label{eq:var}%
\end{equation}
attracted a lot of attention. Let us mention several contributions, beginning
with the classical work \cite{Birk} and followed by the works
\cite{Ding-Zan,Opial}, with numerous other articles that were devoted to this topic (see
\cite{Amb,Conley,Ekel,Fond1,Fond2,Fond3,M-A,M-W1,M-W2,M-T,Offin,Rab1,Rab2,Rab3,Rab4,Serra-Tar1,Serra-Tar2,Will}). It should be pointed out that variational structure of the system
\eqref{eq:var} seems to play crucial role for the application of the
topological and geometric methods.  Regarding the degree theory (cf. \cite{G-M,Mawhin}), it has been
successfully applied to non-Hamiltonian systems in  \cite{B-Z,F}) (see also
\cite{C-R,Kong,W-Q,Z-R}).
\vs
Conditions ($A_{1}$)--($A_{3}$) express the  symmetric properties of the
equation \eqref{eq:sys-non-aut}. Indeed, finding $2\pi m$-periodic solutions to
\eqref{eq:sys-non-aut} leads to an operator which is $D_{m}\times{\mathbb{Z}%
}_{2}$-equivariant. Notice that the ${\mathbb{Z}}_{2}$-action allows us to
make a distinction between constant and non-constant solutions. We do not
require that $f$ is of a gradient-type or has any differentiability
properties, except for the existence of the linearization at $0$.
\vs
Since problem \eqref{eq:sys-non-aut} leads naturally to a
$D_{m}\times{\mathbb{Z}}_{2}$-equivariant equation in functional spaces, one
should ask a question: \textit{what would be the impact of additional
(geometric) symmetries of equation \eqref{eq:sys-non-aut} on the existence and
multiplicity of subharmonic solutions?} Therefore, it is natural to assume
that the system \eqref{eq:sys-non-aut} has additional symmetries
represented by a group $\Gamma$. In this paper, we assume that  $\Gamma$ is  a
finite group acting on vectors in $\mathbb{R}^{k}$ by permuting their
coordinates (see assumption ($A_{4}$)), i.e. the functional equation has the
symmetries
\[
G:=\Gamma\times D_{m}\times{\mathbb{Z}}_{2}.
\]

\vs
We use Brouwer $G$-equivariant degree to establish the existence and
multiplicity of subharmonic ${2\pi m}$-periodic solutions to
\eqref{eq:sys-non-aut}. We make some additional assumptions in order to
illustrate an application of the Brouwer equivariant degree to this systems of
differential equations. First we assume (see the assumption ($A_{5}$)) that
the linearization at $0$ exists and is non-degenerate. We also impose on $f$
the Nagumo growth condition, which implies the existence of \textit{a
priori} bounds on periodic solutions to \eqref{eq:sys-non-aut}.
\vs
We explore in detail two cases of systems of equation: (a) non-symmetric (with
$\Gamma$ being trivial), and (b) with additional symmetries $\Gamma=D_{3}$ and
$\Gamma=D_{5}$. The group $D_{3}$ is the simplest non-abelian group, but it  already makes
a significant  impact on the existence of multiple subharmonic solutions. Since the
computations of Brouwer $G$-equivariant degree can be technically challenging,
in order to overcome these difficulties we use the equivariant degree package
\texttt{EquiDeg} for GAP programming, which was created by Hao-Pin Wu and is
available from \url{https://github.com/psistwu/GAP-equideg}
\vs
As the assumption ($A_{3}$) implies that $f(t,0)=0$, 
so \eqref{eq:sys-non-aut}  admits the (trivial) solution 
 $u(t)=0$. It is interesting to study a parametrized by $\alpha$ modification of the system \eqref{eq:sys-non-aut} (see system \eqref{eq:sys-non-aut-bif}) for which the existence of 
 non-constant branches of subharmonic $2\pi m$-periodic solutions  bifurcating  from $0$ can be analyzed. We 
 apply the Brouwer $G$-equivariant degree method to study the symmetric bifurcation problem for \eqref{eq:sys-non-aut-bif}. We establish 
 the existence of multiple branches  of subharmonic solutions  emerging from the trivial solutions as the parameter $\alpha$ crosses a critical value.
Theoretical results are  illustrated by an example involving concrete symmetries for the system \eqref{eq:sys-non-aut-bif}. 
\vs
\section{Reversible Non-Autonomous Differential Equations}

\label{sec:reversible1}

We are interested in studying the existence of the so-called
\textit{subharmonic} periodic solutions to \eqref{eq:sys-non-aut}, i.e. in
finding non-constant solutions, which for some integer $m\geq3$ satisfy
\begin{equation}
u(t)=u(t+{2\pi m}),\;\;\dot{u}(t)=\dot{u}(t+{2\pi m}).\label{eq:pm-per}%
\end{equation}
We also consider a subgroup $\Gamma\leq S_{k}$  acting on $V:=\mathbb{R}^{k}$ by permuting the coordinates of vectors $x=(x_{1},x_{2},\dots,x_{k}%
)^{T}\in \mathbb{R}^{k}$, i.e. for $\sigma\in S_k$ 
\begin{equation}
\sigma x=\sigma(x_{1},x_{2},\dots,x_{k})^{T}:=(x_{\sigma(1)},x_{\sigma
(2)},\dots,x_{\sigma(k)})^{T}.\label{eq:G-act-V}%
\end{equation}
Clearly, the space $V:=\mathbb{R}^{k}$ equipped with this $\Gamma$-action is
an orthogonal $\Gamma$-represen\-tation. As it is also of our interest to
study the impact of the symmetries $\Gamma$ on the existence of subharmonic
solutions to \eqref{eq:sys-non-aut}, we introduce the following condition:

\begin{itemize}
 \itemindent=2pt\labelsep=3pt\labelwidth5pt\itemsep=1pt

\item[($A_{4}$)] ~~ For all $t\in\mathbb{R}$, $x\in\mathbb{R}^{k}$ and
$\sigma\in\Gamma$, we have $f(t,\sigma x)=\sigma f(t,x)$.
\end{itemize}

The condition ($A_{4}$) implies that $f$ is $\Gamma$-equivariant, i.e. the
system \eqref{eq:sys-non-aut} admits $\Gamma$-symmetries. \vskip.3cm

\subsection{\textbf{Reformulation of \eqref{eq:sys-non-aut} in Functional
Spaces}}

Consider the Banach space $\mathbb{F}:=C_{2\pi m}(\mathbb{R},V)$ of all $2\pi
m$-periodic continuous $V$-valued functions with the usual sup-norm
\[
\|\varphi\|_{\infty}:= \max_{t\in\mathbb{R}} |\varphi(t)|, \; \varphi
\in\mathbb{F},
\]
and denote by $\mathbb{E}:=C^{2}_{2\pi m}(\mathbb{R},V)$ the Banach space of
all $2\pi m$-periodic $C^{2}$-differentiable $V$-valued functions with the
norm $\|\cdot\|:=\|\cdot\|_{2,\infty}$ given by
\begin{equation}
\label{eq:sup-norm2}\|u\|=\|u\|_{2,\infty}:=\max\{ \|u\|_{\infty}, \|\dot u
\|_{\infty},\|\ddot u\|_{\infty}\}, \quad u\in\mathbb{E}.
\end{equation}

Notice that the natural injection operator $\mathfrak{j}:\mathbb{E}%
\rightarrow\mathbb{F}$, $(\mathfrak{j}(u))(t):=u(t)$, $t\in\mathbb{R}$, is a
compact linear operator. We define the operator $L:\mathbb{E}\rightarrow
\mathbb{F}$ by $L(u)(t):=\ddot{u}(t)-u(t)$, $u\in\mathbb{E}$, and the
continuous map $N_{f}:\mathbb{F}\rightarrow\mathbb{F}$ by $N_{f}%
(\varphi)(t)=f(t,\varphi(t))$, $\varphi\in\mathbb{F}$. Then, the system
\eqref{eq:sys-non-aut} is equivalent to the following operator equation
\begin{equation}
Lu=N_{f}(\mathfrak{j}(u))-\mathfrak{j}(u),\quad u\in\mathbb{E}%
.\label{eq:op-eq}%
\end{equation}
Since the operator $L$ is an isomorphism, we can rewrite \eqref{eq:op-eq} as
\[
u=L^{-1}\Big(N_{f}(\mathfrak{j}(u))-\mathfrak{j}(u)\Big),\quad u\in\mathbb{E}.
\]
Define the map $\mathscr F:\mathbb{E}\rightarrow\mathbb{E}$, by
\begin{equation}
\mathscr F(u):=u-L^{-1}\Big(N_{f}(\mathfrak{j}(u))-\mathfrak{j}(u)\Big),\quad
u\in\mathbb{E}.\label{eq:non-aut-F}%
\end{equation}
Then $u\in\mathbb{E}$ is a solution to \eqref{eq:sys-non-aut} if and only if
\begin{equation}
\mathscr F(u)=0.\label{eq:fix-eq}%
\end{equation}
One can easily observe that, by compactness of $\mathfrak{j}$, the map
$\mathscr F$ is a completely continuous field on $\mathbb{E}$. \vskip.3cm
Obviously, by the condition ($A_{3}$), we have $f(t,0)=0$ for all
$t\in\mathbb{R}$, thus $\mathscr F(0)=0$, i.e. the zero function is the
\textit{trivial solution} to \eqref{eq:sys-non-aut}. In what follows we are
interested in finding non-trivial (i.e. non-constant) $2\pi m$-periodic
solutions to \eqref{eq:sys-non-aut}. The group $G:=\Gamma\times D_{m}%
\times{\mathbb{Z}}_{2}$ acts on the space $\mathbb{E}$ by
\begin{align*}
(\sigma,\gamma^{j},\pm1)u(t) &  :=\pm\sigma u(t+2\pi  j),\quad j=0,1,\dots
,m-1,\;\;\sigma\in\Gamma,\;\gamma=e^{\frac{i2\pi}{m}},\\
(\sigma,\kappa,\pm1)u(t) &  :=\pm\sigma u(-t),\quad t\in\mathbb{R}%
,\;\;u\in\mathbb{E},
\end{align*}
thus $\mathbb{E}$ is an isometric Banach $G$-representation. One can easily
verify that the properties ($A_{1}$)---($A_{4}$) imply that $\mathscr F$ is $G
$-equivariant. \vskip.3cm
\vs 
\subsection{\textbf{$G$-Isotypic Decomposition of $\mathbb{E}$}}

Actually, $\mathbb{E}$ is an isometric Banach $\Gamma\times O(2)\times
{\mathbb{Z}}_{2}$-representation, with $O(2)$-action given by
\[
e^{i\theta}u(t)=u\left(  t+\theta m\right)  ,\;\;\kappa u(t)=u(-t),\quad
u\in\mathbb{E},
\]
and $\Gamma$-action given by $(\gamma u)(t)=\gamma u(t)$, $\gamma\in\Gamma$,
$t\in\mathbb{R}$. Using the $\Gamma\times O(2)\times{\mathbb{Z}}_{2}$-action
on $\mathbb{E}$, one can easily recognize the $\Gamma\times D_{m}%
\times{\mathbb{Z}}_{2}$-isotypic decomposition of $\mathbb{E}$. Indeed, by
using the usual Fourier series expansions of functions $u\in\mathbb{E}$, we
have the following $\Gamma\times O(2)\times{\mathbb{Z}}_{2} $-isotypic
decomposition of $\mathbb{E}$
\begin{equation}
\mathbb{E}=\overline{\bigoplus_{j=0}^{\infty}\bigoplus_{l=0}^{\mathfrak{r}%
}\mathbb{V}_{j,l}},\label{eq:G-iso-App2}%
\end{equation}
where
\[
\mathbb{V}_{j,l}=\left\{  u\in\mathbb{E}:u(t)=\cos(jt/m)a+\sin
(jt/m)b,\;a,\,b\in V_{l}\right\}  .
\]
and
\[
V=V_{0}\oplus V_{1}\oplus\dots\oplus V_{\mathfrak{r}},
\]
is a $\Gamma$-isotypic decomposition of $V$, with the component $V_{l}$ being
modeled on the $\Gamma$-irreducible representation $\mathcal{U}_{l}$, $0\leq
l\leq\mathfrak{r}$.

\begin{proposition}
For $j>0$, the $\Gamma\times O(2)\times{\mathbb{Z}}_{2}$-invariant subspace
$\mathbb{V}_{j,l}$ can be identified with the complexification $V_{l}%
^{c}:=V_{l}\oplus iV_{l}$ of $V_{l}$, on which $O(2)$ acts by
\[
e^{i\theta}(a+ib):=e^{-ij\theta}\cdot(a+ib),\;\;\kappa(a+ib)=a-ib,\quad
a,\,b\in V_{l},
\]
where `$\cdot$' stands for complex multiplication.
\end{proposition}

\begin{proof}
Define the real isomorphism $\psi_{j}:V_{l}^{c}\rightarrow\mathbb{V}_{j,l}$ by
$\psi(a+ib)(t)=\cos(jt/m)a+\sin(jt/m)b$, where \ $a$, $b\in V_{l}$. Then for
$\mathbf{z}:=a+ib$ we have
\begin{align*}
\psi_{j}\big(e^{i\theta}(\mathbf{z}\big)) &  =\psi_{j}\big(e^{i\theta
}(a+ib)\big)\\
&  =\psi_{j}\Big(\cos(j\theta)a+\sin(j\theta)b+i(-\sin(j\theta)a+\cos
(j\theta)b)\Big)\\
&  =\cos(jt/m)(\cos(j\theta)a+\sin(j\theta)b)+\sin(jt/m)(-\sin(j\theta
)a+\cos(j\theta)b)\\
&  =\cos(\tfrac{j}{m}(t+m\theta))a+\sin(\tfrac{j}{m}(t+m\theta))b\\
&  =e^{i\theta}(\cos(jt/m)a+\sin(jt/m)b)=e^{i\theta}\psi_{j}(a+ib)=e^{i\theta
}\psi_{j}(\mathbf{z}).
\end{align*}
\vskip.3cm
\end{proof}

Consider $j$-th irreducible $O(2)$-representation $\mathcal{W}_{j}%
\simeq{\mathbb{C}}$, $j>0$, where for $e^{i\theta}\in SO(2)$, $e^{i\theta
}z:=e^{i\theta j}\cdot z$ and $\kappa z:=\overline z$, $z\in{\mathbb{C}}$.
Clearly, since $D_{m} \le O(2)$, $\mathcal{W} _{j}$ is a $D_{m}$-representation.
Put $\mathfrak{s}:=\left\lfloor \frac{m+1}2\right\rfloor $. The irreducible
$D_{m}$-representations $\mathcal{V} _{i}$ are:

\begin{itemize}
\item if $i=0$, then $\mathcal{V}_{0}\simeq\mathbb{R}$ with the trivial
$D_{m}$-action;

\item if $0<i<m/2$, then $\mathcal{V}_{i}\simeq\mathbb{R}^{2}= {\mathbb{C}}$,
where $\gamma z=\gamma^{i}\cdot z$, $\kappa z=\overline z$, $z\in{\mathbb{C}}$;

\item if $i=\mathfrak{s}$, then $\mathcal{V} _{\mathfrak{s}}\simeq\mathbb{R} $
with the $D_{m}$-action $\gamma x=x$, $\kappa x=-x$, $x\in\mathbb{R}$;

\item if $m$ is even, then we have the irreducible $D_{m}$-representation
$\mathcal{V}_{\mathfrak{s}+1}\simeq\mathbb{R}$ with the $D_{m}$-action $\gamma
x=-x$, $\kappa x=x$, $x\in\mathbb{R}$;

\item if $m$ is even, then we have the representation $\mathcal{V}%
_{\mathfrak{s}+2}\simeq\mathbb{R}$ with $D_{m}$-action $\gamma x=-x$, $\kappa
x=-x$, $x\in\mathbb{R}$.
\end{itemize}

For the group $D_{m}\times{\mathbb{Z}}_{2}$,  the corresponding irreducible
representations (with non-trivial ${\mathbb{Z}}_{2}$-action) will be denoted
by $\mathcal{V}_{i}^{-}$.

\begin{proposition}
The $D_{m}$-representation $\mathcal{W}_{j}$ has the following $D_{m}$-isotypic decomposition
\begin{itemize}
\item[$\bullet$] $\mathcal{W} _{mj} \simeq\mathcal{V} _{0}\oplus\mathcal{V}
_{\mathfrak{s}}$,

\item[$\bullet$] for $0<i<\frac m2$, $\mathcal{W} _{mj+i}\simeq\mathcal{W}
_{mj-i}\simeq\mathcal{V} _{i}$,

\item[$\bullet$] if $m$ is even, $\mathcal{W} _{mj-\frac m2} \simeq\mathcal{V}
_{\mathfrak{s}+1} \oplus\mathcal{V} _{\mathfrak{s}+2}$.
\end{itemize}
\end{proposition}
\vs

For $j>0$ and $0\leq l\leq\mathfrak{r}$, we put
\begin{align*}
\mathbb{V}_{jm,l}^{+} &  =\left\{  u\in\mathbb{E}:u(t)=\cos(jt)a,\;a\in
V_{l}\right\}  ,\\
\mathbb{V}_{jm,l}^{-} &  =\left\{  u\in\mathbb{E}:u(t)=\sin(jt)b,\;b\in
V_{l}\right\}  ,\\
\mathbb{V}_{jm-\frac{m}{2},l}^{+} &  =\left\{  u\in\mathbb{E}:u(t)=\cos
((j-\tfrac{1}{2})t)a,\;a\in V_{l}\right\}  ,\\
\mathbb{V}_{jm-\frac{m}{2},l}^{-} &  =\left\{  u\in\mathbb{E}:u(t)=\sin
((j-\tfrac{1}{2})t)b,\;b\in V_{l}\right\}  .
\end{align*}
Therefore, we have the following $\Gamma\times D_{m}\times{\mathbb{Z}}_{2}%
$-isotypic decomposition of the space $\mathbb{E}$:
\[
\mathbb{E}=\bigoplus_{l=0}^{\mathfrak{r}}\bigoplus_{i=0}^{\mathfrak{s}^{\ast}%
}\mathcal{E}_{i,l}^{-},\quad s^{\ast}:=%
\begin{cases}
\mathfrak{s} & \text{ if $m$ is odd}\\
\mathfrak{s}+2 & \text{ if $m$ is even}%
\end{cases}
,\quad\mathfrak{s}=\left\lfloor \frac{m+1}{2}\right\rfloor ,
\]
where
\[
\mathcal{E}_{0,l}^{-}=\mathbb{V}_{0,l}\oplus\overline{\bigoplus_{j=1}^{\infty
}\mathbb{V}_{mj,l}^{+}},\quad\mathcal{E}_{\mathfrak{s},l}^{-}=\overline
{\bigoplus_{j=1}^{\infty}\mathbb{V}_{mj,l}^{-}},
\]
for $0<i<\frac{m}{2}$
\[
\mathcal{E}_{i,l}^{-}=\overline{\bigoplus_{j=0}^{\infty}\mathbb{V}_{mj+i,l}%
}\oplus\overline{\bigoplus_{j=1}^{\infty}\mathbb{V}_{mj-i,l}},
\]
and if $m$ is even then
\[
\mathcal{E}_{\mathfrak{s}+1,l}^{-}=\overline{\bigoplus_{j=1}^{\infty
}\mathbb{V}_{mj-\frac{m}{2},l}^{+}},\quad\mathcal{E}_{\mathfrak{s}+2,l}%
^{-}=\overline{\bigoplus_{j=1}^{\infty}\mathbb{V}_{mj-\frac{m}{2},l}^{-}}.
\]

The component $\mathcal{E}_{i,l}^{-}$ ($0\leq i\leq\mathfrak{s}^{\ast}$,
$0\leq l\leq\mathfrak{r}$) is modeled on the irreducible $\Gamma\times
D_{m}\times{\mathbb{Z}}_{2}$-representation
\[
\mathcal{V}_{i,l}^{-}:=\mathcal{V}_{i}^{-}\otimes\mathcal{U}_{l}.
\]
Since the operator $L$ is $O(2)\times{\mathbb{Z}}_{2}$-equivariant
isomorphism, thus $L(\mathbb{V}_{j})=\mathbb{V}_{j}$ and $L|_{\mathbb{V}_{j}%
}=-(j^{2}/m^{2}+1)\id_{\mathbb{V}_{j}}$. \vskip.3cm

\subsection{\textbf{Linearization of Equation \eqref{eq:sys-non-aut} at $0$
:}}

We make the following additional assumption

\begin{itemize}
\item[($A_{5}$)] ~~ There exists a \textbf{symmetric matrix} $A:\mathbb{R}%
^{k}\rightarrow\mathbb{R}^{k}$ such that
\begin{equation}
\lim_{x\rightarrow0}\frac{f(t,x)-Ax}{|x|}=0\label{eq:lin-1}%
\end{equation}
uniformly with respect to $t\in\mathbb{R}$, and for all integers $j\geq0$ and
$\mu\in\sigma(A)$,
\begin{equation}
\frac{j^{2}}{m^{2}}+\mu\not =0.\label{eq:lin-2}%
\end{equation}

\end{itemize}

We define the linear operator $\mathscr A:\mathbb{E}\rightarrow\mathbb{E}$
(associated with $A:\mathbb{R}^{k}\rightarrow\mathbb{R}^{k}$) by
\begin{equation}
\mathscr Au:=u-L^{-1}\Big(N_{A}(\mathfrak{j}(u)-\mathfrak{j}(u)\Big),\quad
u\in\mathbb{E},\label{eq:Lin-A}%
\end{equation}
where $N_{A}(\varphi)(t):=A(\varphi(t))$, $t\in\mathbb{R}$, $\varphi\in
C_{2\pi m}(\mathbb{R};V)$. Under the assumption ($A_{5}$) the operator
$\mathscr A:\mathbb{E}\rightarrow\mathbb{E}$ given by \eqref{eq:Lin-A} is an
isomorphism and $D\mathscr F(0)=\mathscr A$.

\begin{lemma}
\label{lem:DF(0)} Assume that $f:\mathbb{R}\times\mathbb{R}^{k}\to
\mathbb{R}^{k}$ satisfies the conditions ($A_{1}$)---($A_{5}$). Then there
exists $\varepsilon>0$ such that the $G$-map $\mathscr F:\mathbb{E}%
\to\mathbb{E}$ (given by \eqref{eq:non-aut-F}) and $\mathscr A:\mathbb{E}%
\to\mathbb{E}$ (given by \eqref{eq:Lin-A}) are $\Omega_{\varepsilon}%
$-admissibly $G$-homotopic (here $\Omega_{\varepsilon}:=B_{\varepsilon}(0)$ in
$\mathbb{E}$).
\end{lemma}

\begin{proof}
Define the linear homotopy $\mathfrak{H}:[0,1]\times\mathbb{E}\to\mathbb{E} $
as $\mathfrak{H}(\lambda,u):=(1-\lambda)\mathscr Au+\lambda\mathscr F(u)$,
$u\in\mathbb{E}$, and suppose for contradiction that there exists a sequence
$\{\lambda_{n},u_{n}\}$ such that $u_{n}\not =0$, $\lambda_{n}\to\lambda_{o}$
and $u_{n}\to0$ as $n\to\infty$\,, then we have
\begin{align*}
0 & =\mathfrak{H}(\lambda_{n},u_{n})=(1-\lambda_{n})\mathscr Au_{n}%
+\lambda_{n} F(u_{n})\\
& =\mathscr A(u_{n})+\lambda_{n}(\mathscr F(u_{n})-\mathscr Au_{n}).
\end{align*}
Put $v_{n}:=\frac{u_{n}}{\|u_{n}\|}$. Then
\[
0=\mathscr A v_{n} +\lambda_{n}\frac{\mathscr F(u_{n})-\mathscr A u_{n}%
}{\|u_{n}\|}.
\]
Since $\|u_{n}\|\to0$ and $\lambda_{n}$ is bounded, thus
\[
\lim_{n\to\infty} \frac{\mathscr F(u_{n})-\mathscr A u_{n}}{\|u_{n}\|} =0,
\]
which implies
\[
0=\lim_{n\to\infty}(\mathscr A v_{n}).
\]
On the other hand, since $\mathscr A=\id -\mathscr K$, where
$\mathscr K:=L^{-1}(A\mathfrak{j}-\mathfrak{j})$ is a compact operator, one
can assume (by passing to a subsequence) that $\mathscr K v_{n}\to v_{o}$,
which implies $v_{n}\to v_{0}$ and $\|v_{o}\|=1$, so $v_{o}\in
\text{\textrm{Ker\,}} \mathscr A$, but this is a contradiction with ($A_{5}$).
\end{proof}

\vs

\subsection{\textbf{Nagumo Growth Condition:}}

The following condition is often referred to as the \textit{Nagumo growth
condition}:

\begin{itemize}
\item[($A_{6}$)] ~~ There exists a constant $M>0$ such that
\[
\forall_{t\in\mathbb{R}}\;\; \forall_{x\in\mathbb{R}^{k}} \;\; |x|\ge M\;\;
\Rightarrow\;\; f(t,x)\bullet x>0.
\]

\end{itemize}
\vs

We consider the following parametrized (by $\lambda\in\lbrack0,1]$)
modification of system \eqref{eq:sys-non-aut}:
\begin{equation}%
\begin{cases}
\ddot{u}(t)=\lambda f(t,u(t))+(1-\lambda)u(t),\quad t\in\mathbb{R},\;u(t)\in
V,\\
u(t)=u(t+{2\pi m}),\;\dot{u}(t)=\dot{u}(t+{2\pi m}).
\end{cases}
\label{eq:non-aut-lambda}%
\end{equation}
\vs

Then we have:

\begin{lemma}
\label{lem-Nagumo} Assume that $f:\mathbb{R}\oplus\mathbb{R}^{k}\to
\mathbb{R}^{k}$ is a continuous function satisfying conditions ($A_{1}%
$)---($A_{4}$) and ($A_{6}$). If $u(t)$ is a ${2\pi m}$-periodic function of
class $C^{2}$ such that $\max_{t\in\mathbb{R}}|u(t)|\ge M$ (where $M$ is given
in ($A_{6}$)), then $u(t)$ cannot be a solution of \eqref{eq:non-aut-lambda}
for $\lambda\in[0,1]$.
\end{lemma}

\begin{proof}
Assume for the contradiction that $u(t)$ is a solution while $\max
_{t\in\mathbb{R}}|u(t)|\ge M$. Consider the function $\phi(t):=\frac{1}{2}
|u(t)|^{2} $. Suppose that $\phi(t_{0})=\max_{t\in\mathbb{R}}\phi(t)$, then
$\phi^{\prime}(t_{0})=u(t_{0})\bullet\dot u(t_{0})=0$ and $\phi^{\prime\prime
}(t_{0})=\dot u(t_{0})\bullet\dot u(t_{0})+ \ddot u(t_{0})\bullet u(t_{0}%
)\leq0$. However, by condition ($A_{6}$), $\phi^{\prime\prime}(t_{0})=\dot
u(t_{0})\bullet\dot u(t_{0})+\ddot u(t_{0})\bullet u(t_{0})= (\lambda
(f(u(t_{0}))-u(t_{0}))+u(t_{0}))\bullet u(t_{0})+ \dot u(t_{0})\bullet\dot
u(t_{0})>(1-\lambda)u(t_{0})\bullet u(t_{0})+ \lambda f(u(t_{0}))\bullet
u(t_{0})>0$, which leads to a contradiction with condition ($A_{6}$).
\end{proof}

\vskip.3cm

\begin{lemma}
\label{lem-Nagumo-lambda} Assume that $f:\mathbb{R}\oplus\mathbb{R}^{k}%
\to\mathbb{R}^{k}$ is a continuous function satisfying conditions ($A_{1}%
$)---($A_{3}$) and ($A_{6}$). Then there exists $R>0$ such that for every
solution $u\in\mathbb{E}$ to $\eqref{eq:non-aut-lambda}$, $\lambda\in[0,1]$,
we have $\|u\|<R$. In addition, for $\Omega_{R}:=B_{R}(0)$, the map
$\mathscr F:\mathbb{E}\to\mathbb{E}$ is $\Omega_{R}$-admissibly $G$-homotopic
to $\id$.
\end{lemma}

\begin{proof}
By Lemma $\ref{lem-Nagumo}$, there exists a $M>0$ such that any ${2\pi m}%
$-periodic solution $u(t)$ to \eqref{eq:non-aut-lambda} satisfies $|u(t)|<M$.
Take $A_{R}:=\{(t,x)\in\lbrack0,{2\pi m}]\times\mathbb{R}^{k}:|x|\leq M\}$.
Since the function $F:[0,1]\times\mathbb{R}\oplus\mathbb{R}^{k}\rightarrow
\mathbb{R}^{k}$ given by
\[
F(\lambda,t,x)=\lambda f(t,x)+(1-\lambda)x,\;\;\;\;x\in\mathbb{R}%
^{k},\;\lambda\in\lbrack0,1]
\]
is continuous, and the set $[0,1]\times A_{R}$ is compact, then for every
solution $u(t)$ to \eqref{eq:non-aut-lambda} we have
\[
|\ddot{u}(t)|=|F(\lambda,t,u)|\leq\sup\{|F(\lambda,t,x)|:(t,x)\in
A_{R},\,\lambda\in\lbrack0,1]\}=:M_{2}.
\]
Put $\dot{u}(t)=(u_{1}^{\prime}(t),u_{2}^{\prime}(t),\dots,u_{k}^{\prime
}(t))^{T}$. Then for every $1\leq l\leq k$, since the function $\dot{u}%
_{l}(t)$ is periodic, there exists $\tau_{o}\in\lbrack0,{2\pi m}]$ such that
$u_{l}^{\prime}(\tau_{o})=0$. Thus the identity
\[
u_{l}^{\prime}(t)=\int_{\tau_{o}}^{t}u_{l}^{\prime\prime}(s)ds,\quad
t\in\mathbb{R}%
\]
implies $|u_{l}^{\prime}(t)|\leq{2\pi m}M_{2}$ for $t\in\mathbb{R}$ and
consequently
\[
\Vert\dot{u}\Vert_{\infty}=\max_{t\in\mathbb{R}}\sqrt{|u_{1}^{\prime}%
(t)|^{2}+|u_{2}^{\prime}(t)|^{2}+\dots+|u_{k}^{\prime}(t)|^{2}}\leq\sqrt
{k}{2\pi m}M_{2}=:M_{1}.
\]
Consequently,
\begin{equation}
\Vert\dot{u}\Vert=\max\{\Vert u\Vert_{\infty},\Vert\dot{u}\Vert_{\infty}%
,\Vert\ddot{u}\Vert_{\infty}\}\leq\max\{M,M_{1},M_{2}\}<\max\{M,M_{1}%
,M_{2}\}+1=:R,
\end{equation}
and the conclusion follows.
\end{proof}

\vskip.3cm

\subsection{Abstract Existence Result}

Assume that $f:\mathbb{R}\oplus V\to V$ satisfies the assumptions ($A_{1}%
$)---($A_{6}$). We denote the set of negative eigenvalues of the operator
$\mathscr A $ by $\sigma_{-}(\mathscr A)$. Then by Lemma \ref{lem:DF(0)},
there exists a sufficiently small $\varepsilon>0$ such that $\mathscr F$ is
$\Omega_{\varepsilon}$-admissibly $G$-homotopic to $\mathscr A$ (given by
\eqref{eq:Lin-A}) and therefore%

\begin{equation}
G\text{\textrm{-deg}}(\mathscr F,\Omega_{\varepsilon})=G\text{\textrm{-deg}%
}(\mathscr A,B(\mathbb{E}))=\prod_{\lambda\in\sigma_{-}(\mathscr A)}%
G\text{\textrm{-deg}}(-\id|_{E(\lambda)},B(E(\lambda
))),\label{eq:prod-for-non}%
\end{equation}
where $E(\lambda)$ denotes the  eigenspace of $\mathscr A$
corresponding to $\lambda$, and $B(E(\lambda))$ stands for an open unit ball
in $E(\lambda)$.
\vs

In order to use the formula \eqref{eq:prod-for-non} we need to compute the
negative spectrum $\sigma_{-}(\mathscr A)$. Since $\mathscr A$ is
$\Gamma\times O(2)\times{\mathbb{Z}}_{2}$-equivariant, one can use the
isotypic decomposition \eqref{eq:G-iso-App2} in order to determine eigenvalues
(and eigenspaces) of $\mathscr A$:
\begin{equation}
\sigma(\mathscr A)=\left\{  \lambda_{j,\mu}:=1+\frac{m^{2}(\mu-1)}{j^{2}%
+m^{2}}:j=0,1,2,\dots,\;\mu\in\sigma(A)\right\}  .
\end{equation}
Clearly,
\[
\lambda_{j,\mu}=\frac{j^{2}+m^{2}\mu}{j^{2}+m^{2}}<0
\]
if and only if $\mu<-j^{2}/m^{2}$. Notice that, in such a case we also have
\[
\lambda_{0,\mu}<\lambda_{1,\mu}<\dots<\lambda_{j-1,\mu}<\lambda_{j,\mu}%
<\dots<\lambda_{\mathfrak{j}_{\mu},\mu}<0<\lambda_{\mathfrak{j}_{\mu}+1,\mu},
\]
where $\mathfrak{j}_{\mu}$ is the integer number (by condition ($A_{5}$))
satisfying
\[
-\frac{(\mathfrak{j}_{\mu}+1)^{2}}{m^{2}}<\mu<-\frac{\mathfrak{j}_{\mu}^{2}%
}{m^{2}}.
\]

On the other hand, by Lemma \ref{lem-Nagumo-lambda}, there exists a
sufficiently large $R>0$ such that $\mathscr F$ is $\Omega_{R}$-admissibly
$G$-homotopic to $\id$. Therefore, $G\text{\textrm{-deg}}(\mathscr F,\Omega
_{R})=G\text{\textrm{-deg}}(\id,\Omega_{R})=(G)$. Put $\Omega:=\Omega
_{R}\setminus\overline{\Omega_{\varepsilon}}$. Then the $G\text{\textrm{-deg}%
}(\mathscr F,\Omega)$ is well defined and by additivity property we have
\begin{align*}
G\text{\textrm{-deg}}(\mathscr F,\Omega) &  =G\text{\textrm{-deg}%
}(\mathscr F,\Omega_{R})-G\text{\textrm{-deg}}(\mathscr F,\Omega_{\varepsilon
})\\
&  =(G)-G\text{\textrm{-deg}}(\mathscr A,B(\mathbb{E})).
\end{align*}
In this way we can formulate the following abstract existence result: \vskip.3cm

\begin{theorem}
\label{th:abst-rev-non-aut} Assume that $f:\mathbb{R}\oplus V\to V$ satisfies
the assumptions ($A_{1}$)---($A_{6}$), $R>0$ is a sufficiently large (given by
Lemma \ref{lem-Nagumo-lambda}), $\varepsilon>0$ is sufficiently small (given
by Lemma \ref{lem:DF(0)}) and $\Omega:=\Omega_{R}\setminus\overline
{\Omega_{\varepsilon}}$. If the $G$-equivariant degree
\[
G\text{\textrm{-deg}}(\mathscr F,\Omega)=n_{1}(H_{1})+n_{2}(H_{2})+\dots+n_{s}
(H_{s}) \in A(G)
\]
has a non-zero coefficient $n_{j}$, then there exists a ${2\pi m}$-periodic
solution $u\in\Omega$ to \eqref{eq:sys-non-aut} such that $G_{u}\le H_{j}$. In
addition, if $D_{m}\not \le H_{j}$ then $u$ is non-constant, and if for some
$g\in D_{m}$, $g\not =1$, we have $(g,-1)\in H_{j}$, then the solution $u$ can
not be $2\pi$-periodic solution, i.e. its minimal period is not $2\pi$.
\end{theorem}

\begin{proof}
The existence of a ${2\pi m}$-periodic solution $x$ is a direct consequence of
the existence property for $G$-equivariant degree. Moreover, if $u(t)$ is
constant, then clearly $u(t+l2\pi)=u(t)$ and $u(-t)=u(t)$, for $t\in\mathbb{R}$
and $l\in{\mathbb{Z}}$, so $D_{m}\le H_{j}$. Assume that there exists an
element $(g,-1)\in H_{j}$ for some $1\not =g\in D_{m}$, which implies that for
some $1\le l\le m-1$ we have $g=\gamma^{l}$ or $g=\gamma^{l}\kappa$. Then we
also have
\[
\forall_{t\in\mathbb{R}} \;\; ((g,-1)u)(t)=u(t) \;\; \Rightarrow x(0)=-x(2\pi l).
\]
Since $x\not \equiv 0$ it follows that $u(t)\not =u(t+l2\pi)$ and consequently
$u(t)\not =u(t+2\pi)$.
\end{proof}

\vskip.3cm

\subsection{Subharmonic Solutions in Non-Equivariant Case}

To illustrate our previous theorem, in this section we assume that $\Gamma=\{e\}$, i.e. we consider the case of
\eqref{eq:sys-non-aut} without additional symmetries and $G=D_{m}%
\times{\mathbb{Z}}_{2}$.

\begin{definition}
We define
\begin{equation}
i(j):=%
\begin{cases}
\alpha(j) & \text{ if }\;\alpha(j)\leq\lfloor\frac{m}{2}\rfloor,\\
m-\alpha(j) & \text{ if }\;\alpha(j)>\lfloor\frac{m}{2}\rfloor,
\end{cases}
\label{eq:j-i}%
\end{equation}
where $\alpha(j)\in\{0,1,\dots,m-1\}$ satisfies $\alpha(j)\equiv j\text{ (mod
$m$)}$, i.e.
\[
\alpha(j):=j-\left\lfloor \frac{j}{m}\right\rfloor m\in\{0,1,2,\dots,m-1\}.
\]

\end{definition}

We use notation $\mathfrak{m}(\mu)$ for the algebraic multiplicity of $\mu$
belonging to the spectrum of $A$. The negative spectrum $\sigma_{-}(\mathscr
A)$ can be represented as
\begin{equation}
\sigma_{-}(\mathscr A)=\bigcup_{\mu\in\sigma_{-}(A)}\{\lambda_{0,\mu}%
,\lambda_{1,\mu},\dots,\lambda_{\mathfrak{j}_{\mu}-1,\mu},\lambda
_{\mathfrak{j}_{\mu},\mu}\}\label{eq:negA}%
\end{equation}
Denote by $E(\lambda_{j,\mu})$ the eigenspace of $\lambda_{j,\mu}$. In order
to compute the degree, we introduce the following notation:
\[
\beta_{0}(\mu):=\left(  \left\lfloor \frac{\mathfrak{j}_{\mu}}{m}\right\rfloor
+1\right)  \mathfrak{m}(\mu),\quad\beta_{\mathfrak{s}}(\mu):=\left\lfloor
\frac{\mathfrak{j}_{\mu}}{m}\right\rfloor \mathfrak{m}(\mu),
\]
and for $i=\mathfrak{s}+1$, $\mathfrak{s}+2$ (in the case $m$ is even),
\[
\beta_{i}(\mu):=%
\begin{cases}
\left\lfloor \frac{\mathfrak{j}_{\mu}}{m}\right\rfloor \mathfrak{m}(\mu), &
\text{ if }\;\alpha(\mathfrak{j}_{\mu})<\frac{m}{2},\\
\left(  \left\lfloor \frac{\mathfrak{j}_{\mu}}{m}\right\rfloor +1\right)
\mathfrak{m}(\mu), & \text{ if }\;\alpha(\mathfrak{j}_{\mu})\geq\frac{m}{2},
\end{cases}
\]
and finally for $0<i<\frac{m}{2}$ ,
\[
\beta_{i}(\mu):=%
\begin{cases}
2\left\lfloor \frac{\mathfrak{j}_{\mu}}{m}\right\rfloor \mathfrak{m}(\mu) &
\text{ if }\alpha(\mathfrak{j}_{\mu})<i,\\
\left(  2\left\lfloor \frac{\mathfrak{j}_{\mu}}{m}\right\rfloor +1\right)
\mathfrak{m}(\mu) & \text{ if }\;i\leq\alpha(\mathfrak{j}_{\mu})<m-i,\\
2\left(  \left\lfloor \frac{\mathfrak{j}_{\mu}}{m}\right\rfloor +1\right)
\mathfrak{m}(\mu) & \text{ if }\;m-i\leq\alpha(\mathfrak{j}_{\mu}).
\end{cases}
\]

\vs
\begin{definition}\rm
We define
\[
\eta_{i}:=\sum_{\mu\in\sigma_{-}(A)}\beta_{i}(\mu),
\]
for $i=0,1,\dots,\mathfrak{s},\mathfrak{s}+1,\mathfrak{s}+2$. The number
$\eta_{i}$ counts the "total number of times" that the irreducible
representation $\mathcal{V}_{i}^{-}$ appears in $\sigma_{-}(\mathscr A)$.
\end{definition}

We have the following list of basic degrees for the irreducible $G$%
-representations (see appendix):

\begin{itemize}
\item for $0\leq i\leq\lfloor\frac{m}{2}\rfloor$, $h:=\text{gcd}(m,i)$, $m/h$
is odd then
\[
\deg_{\mathcal{V}_{i}^{-}}=(D_{m}\times{\mathbb{Z}}_{2})-(D_{h})-(D_{h}%
^{z})+({\mathbb{Z}}_{h});
\]

\item if $m/h\equiv2$ (mod $4$) then
\[
\deg_{\mathcal{V}_{i}^{-}}=(D_{m}\times{\mathbb{Z}}_{2})-(D_{2h}^{d}%
)-(D_{2h}^{\hat{d}})+({\mathbb{Z}}_{2h}^{d});
\]

\item if $m/h\equiv0$ (mod $4$) then
\[
\deg_{\mathcal{V}_{i}^{-}}=(D_{m}\times{\mathbb{Z}}_{2})-(D_{2h}%
^{d})-(\widetilde{D}_{2h}^{d})+({\mathbb{Z}}_{2h}^{d});
\]

\item if $i=\mathfrak{s}$ then
\[
\deg_{\mathcal{V}_{\mathfrak{s}}^{-}}=(D_{m}\times{\mathbb{Z}}_{2})-(D_{m}%
^{z});
\]

\item if $i=0$ then
\[
\deg_{\mathcal{V}_{0}^{-}}=(D_{m}\times{\mathbb{Z}}_{2})-(D_{m});
\]

\item if $m$ is even and $i=\mathfrak{s}+1$ then
\[
\deg_{\mathcal{V}_{\mathfrak{s}+1}^{-}}=(D_{m}\times{\mathbb{Z}}_{2}%
)-(D_{m}^{d});
\]

\item if $m$ is even and $i=\mathfrak{s}+2$ then
\[
\deg_{\mathcal{V}_{\mathfrak{s}+2}^{-}}=(D_{m}\times{\mathbb{Z}}_{2}%
)-(D_{m}^{\hat{d}}).%
\]

\end{itemize}
\vs

Notice that $\deg_{\mathcal{V}_{i}^{-}}=\deg_{\mathcal{V}_{i^{\prime}}^{-}}$
if and only if gcd$(i,m)=\text{gcd}(i^{\prime},m)$. Therefore, we introduce
the numbers $\rho_{i}$, $0\leq i\leq\mathfrak{s}+2$, that will allow us to
determine how many times the basic degree $\deg_{\mathcal{V}_{i}^{-}}$ appears
in the degree of $\gdeg(\mathscr A,B(\mathbb{E}))$.

\begin{definition}\rm 
We define $\rho_{0}:=\eta_{0},\;\rho_{\mathfrak{s}}:=\eta_{\mathfrak{s}%
},\;\rho_{\mathfrak{s}+1}:=\eta_{\mathfrak{s}+1},\;\rho_{\mathfrak{s}+2}%
:=\eta_{\mathfrak{s}+2},$ and
\begin{equation}
\rho_{i}:=\sum_{\text{gcd}(i^{\prime},m)=\text{gcd}(i,m)}\eta_{i^{\prime}%
},\qquad0<i<\frac{m}{2}.\label{eq:numb=rho-i}%
\end{equation}

\end{definition}

Before proving our main theorem, we need to analyze the maximal $G$-orbit
types in the space $\mathbb{E}\setminus\{0\}$.

\begin{lemma}
\label{lem:max-orbG} Suppose $m= 2^{\mathfrak{n}} m^{\prime}$, where
$m^{\prime}$ is an odd integer. Then the maximal orbit types in $\mathbb{E}%
\setminus\{0\}$ are:

\begin{itemize}
\item[(a)] $(D_{m}^{z})$, $(D_{m})$, and if $\mathfrak{n}>0$,

\item[(b)] $(D_{m}^{d})$, $(\widetilde{D}_{m}^{d})$, $(D_{\frac{m}{2}}^{d})$,
$(\widetilde{D}_{\frac{m}{2}}^{d})$,\dots, $(D_{\frac{m}{2^{\mathfrak{n}-1}}%
}^{d})$, $(\widetilde{D}_{\frac{m}{2^{\mathfrak{n}-1}}}^{d})$,
\end{itemize}
\end{lemma}

\begin{proof}
The maximal $G$-orbit types in $\mathbb{E}\setminus\{0\}$ are exactly the same
as the maximal $G$-orbit types which occur in the space $\mathcal{V}%
^{*}\setminus\{0\}$, where
\[
\mathcal{V}^{*}:=\mathcal{V} _{0}^{-}\oplus\mathcal{V} _{1}^{-}\oplus
\mathcal{V} _{2}^{-}\oplus\mathcal{V} ^{-}_{\mathfrak{s}}\oplus\mathcal{V}
^{-}_{\mathfrak{s}+1}\oplus\mathcal{V} ^{-}_{\mathfrak{s}+2}.
\]
First we identify the maximal orbit types in $\mathcal{V} _{i}^{-}%
\setminus\{0\}$, $i=0,1,\dots, \mathfrak s+2$,
\begin{itemize}
\item for $0\leq i\leq\lfloor\frac{m}{2}\rfloor$, $h:=\text{gcd}(m,i)$,
$p_{h}:=\frac{m}{h}$ and if $p_{h}$ is odd, then the orbit types are: $(D_{h})
$, $(D_{h}^{z})$;

\item if $p_{h}\equiv2$ (mod $4$), then the maximal orbit types are:
$((D_{2h}^{d})$, $(D_{2h}^{\hat{d}})$;

\item if $p_{h}\equiv0$ (mod $4$), then the maximal orbit types are:
$(D_{2h}^{d})$, $(\widetilde{D}_{2h}^{d})$;

\item if $i=\mathfrak{s,}$ then the maximal orbit type is: $(D_{m}^{z})$;

\item if $i=0,$ then the maximal orbit type is: $(D_{m})$;

\item if $m$ is even and $i=\mathfrak{s}+1,$ then the maximal orbit type is:
$(D_{m}^{d})$;

\item if $m$ is even and $i=\mathfrak{s}+2,$ then the maximal orbit type is:
$(D_{m}^{\hat{d}})$.
\end{itemize}

Notice that $(D_{m})$ and $(D_{m}^{z})$ are the maximal orbit types which
occur in $\mathcal{V} _{0}^{-}\oplus\mathcal{V} _{\mathfrak{s}}^{-}%
\setminus\{0\}$. On the other hand for $m$ being an even integer, we have (see
Table 5.3 in \cite{AED}):

\begin{itemize}
\item $(D_{2n}^{d}\le(D_{m}^{d})$ if and only if $n|\frac m2$ and $\frac
m{2n}$ is odd;

\item $(\widetilde D_{2n}^{d}\le(\widetilde D_{m}^{d})$ if and only if
$n|\frac m2$ and $\frac m{2n}$ is odd;

\item $(D_{2n}^{\hat{d}}\leq(D_{m}^{\hat{d}})$ if and only if $n|\frac{m}{2} $
and $\frac{m}{4n}$ is odd;
\end{itemize}
\noi and 
the maximality of the orbit types $(D_{m}^{d})$, $(\widetilde{D}_{m}^{d})$,
$(D_{\frac{m}{2}}^{d})$, $(\widetilde{D}_{\frac{m}{2}}^{d})$,\dots,
$(D_{\frac{m}{2^{\mathfrak{n}-1}}}^{d})$, $(\widetilde{D}_{\frac
{m}{2^{\mathfrak{n}-1}}}^{d})$ follows from this result.
\end{proof}
\vs
We have the main theorem.
\vs

\begin{theorem}
\label{th:non-auto} Let $m$ be a natural number and $f:\mathbb{R}%
\times\mathbb{R}^{k}\rightarrow\mathbb{R}^{k}$ be a continuous function
satisfying the assumptions \textrm{($A_{1}$)---($A_{3}$)} and \textrm{($A_{5}
$)---($A_{6}$)}. Suppose $m=2^{\varepsilon_{0}}p_{1}^{\varepsilon_{1}}%
p_{2}^{\varepsilon_{2}}\dots p_{s}^{\varepsilon_{s}}$,
where $\varepsilon_{0}\geq0$, $\varepsilon_{l}>0$, and $p_{l}$, $l=1,2,\dots
,s$, are the prime numbers such that $2<p_{1}<p_{2}<\dots<p_{s}$.
For $l=1,2,\dots s$ put $m_{l}:=\frac{m}{p_{l}}$. Then

\begin{itemize}
\item[(i)] ~~ if $\rho_{0}$ is odd, then the system \eqref{eq:sys-non-aut}
admits a $G$-orbit of ${2\pi m}$-periodic solutions with symmetries $(D_{m})
$, and

\item[(ii)] ~~ if $\rho_{\mathfrak{s}}$ is odd, then the system
\eqref{eq:sys-non-aut} admits a $G$-orbit of ${2\pi m}$-periodic solutions
with symmetries $(D^{z}_{m})$, and

\item[(iii)] ~~ if for some $l=1,2,\dots,s$, $\rho_{m_{l}}$ is odd, then the
system \eqref{eq:sys-non-aut} admits a $G$-orbit of ${2\pi m}$-periodic
solutions with symmetries either $(D_{m_{i}}^{z})$ or $(D_{m}^{z})$, and

\item[(iv)] ~~ if $\varepsilon_{0}>0$, and $\rho_{\mathfrak{s}+1}$ is odd,
then the system \eqref{eq:sys-non-aut} admits a $G$-orbit of ${2\pi m}%
$-periodic solutions with symmetries exactly $(D_{ m}^{d})$, and

\item[(iv)] ~~ if $\varepsilon_{0}>0$, and $\rho_{\mathfrak{s}+2}$ is odd,
then the system \eqref{eq:sys-non-aut} admits a $G$-orbit of ${2\pi m}%
$-periodic solutions with symmetries exactly $(D_{m}^{\hat d})$, and

\item[(iv)] ~~ if $\varepsilon_{0}=\mathfrak{k}>0$, and for some $2\leq
k\leq\mathfrak{k}$, $\rho_{\frac{m}{2^{k}}}$ is odd, then the system
\eqref{eq:sys-non-aut} admits a $G$-orbit of ${2\pi m}$-periodic solutions
with symmetries exactly $(D_{\frac{m}{2^{k-1}}}^{d})$ and $(D_{\frac
{m}{2^{k-1}}}^{\hat{d}})$.
\end{itemize}
\end{theorem}

\begin{proof}
Clearly for $\mu\in\sigma_{-}(A)$ we have the following formula for the
$\mathcal{V}_{i}^{-}$-isotypic multiplicity of the eigenvalue $\lambda_{j,\mu
}$
\begin{equation}
m_{i}^{-}(\lambda_{j,\mu})=%
\begin{cases}
\mathfrak{m}(\mu) & \text{ if }\;i=i(j),\;0<i<\frac{m}{2},\\
\mathfrak{m}(\mu) & \text{ if }\;i=\mathfrak{s},\;j>0,\;i(j)=0,\\
\mathfrak{m}(\mu) & \text{ if }\;i=\mathfrak{s}+1,\;i(j)=\frac{m}{2},\\
\mathfrak{m}(\mu) & \text{ if }\;i=\mathfrak{s}+2,\;i(j)=\frac{m}{2},\\
0 & \text{ otherwise}.
\end{cases}
\label{eq:lamb-mult}%
\end{equation}
Therefore (by \eqref{eq:lamb-mult}), we have
\begin{equation}
G\text{\textrm{-deg}}(\mathscr A,B(\mathbb{E}))=\prod_{\mu\in\sigma_{-}%
(A)}\prod_{j=0}^{\mathfrak{j}_{\mu}}G\text{\textrm{-deg}}(-\id,B(E(\lambda
_{j,\mu})),\label{eq:prod-for-non2}%
\end{equation}
where
\begin{equation}
G\text{\textrm{-deg}}(-\id,B(E(\lambda_{j,\mu}))=%
\begin{cases}
(\deg_{\mathcal{V}_{0}^{-}})^{\mathfrak{m}(\mu)} & \text{ if }\;j=0,\\
(\deg_{\mathcal{V}_{0}^{-}})^{\mathfrak{m}(\mu)}\circ(\deg_{\mathcal{V}%
_{\mathfrak{s}}^{-}})^{\mathfrak{m}(\mu)} & \text{ if }\;j>0,\;i(j)=0,\\
(\deg_{\mathcal{V}_{i}^{-}})^{\mathfrak{m}(\mu)} & \text{ if }\;0<i=i(j)<\frac
{m}{2},\\
(\deg_{\mathcal{V}_{\mathfrak{s}+1}^{-}})^{\mathfrak{m}(\mu)}\circ
(\deg_{\mathcal{V}_{\mathfrak{s}+2}^{-}})^{\mathfrak{m}(\mu)} & \text{ if
}\;i(j)=\frac{m}{2}.
\end{cases}
\label{eq:eign-deg}%
\end{equation}
Consequently, we obtain the following formula:
\begin{align*}
G\text{\textrm{-deg}}(\mathscr F,\Omega)  & =(G)-G\text{\textrm{-deg}%
}(\mathscr A,B(\mathbb{E}))\\
& =(G)-\prod_{\mu\in\sigma_{-}(A)}\prod_{i=0}^{\mathfrak{s}+2}(\deg
_{\mathcal{V}_{i}^{-}})^{\beta_{i}(\mu)}.
\end{align*}

Since $\eta_{i}$ counts the "total number of times" that the irreducible
representation $\mathcal{V}_{i}^{-}$ appears in $\sigma_{-}(\mathscr A)$, we
obtain that
\begin{equation}
G\text{\textrm{-deg}}(\mathscr F,\Omega)=(G)-\prod_{i=0}^{\mathfrak{s}+2}%
(\deg_{\mathcal{V}_{i}^{-}})^{\eta_{i}}.\label{eq:lin-form-non}%
\end{equation}
Since
\[
\deg_{\mathcal{V}_{i}^{-}}=\deg_{\mathcal{V}_{i^{\prime}}^{-}}\text{\qquad
}\Leftrightarrow\text{\qquad gcd}(i,m)=\text{gcd}(i^{\prime},m),
\]
and the numbers $\rho_{i}$ indicate the number of occurrences of the basic
degree $\deg_{\mathcal{V}_{i}^{-}}$ in the product \eqref{eq:lin-form-non},
the conclusion follows from Lemma \ref{lem:max-orbG} and the fact that the
square of any basic degree is the unit element $(G)\in A(G)$, i.e. $(\deg
_{\mathcal{V}_{i}^{-}})^{2}=(G).$
\end{proof}
\vs
In order to illustrate the applications of Theorem \ref{th:non-auto}, we only
consider a simple case when the operator $A$ has one simple negative
eigenvalue $\mu$ satisfying
\begin{equation}
\label{eq:simple-ex}-\frac{(\mathfrak{p}+1)^{2}}{m^{2}}<\mu<-\frac
{\mathfrak{p}^{2}}{m^{2}}, \quad\mathfrak{p}:=\left\lfloor \frac{m}2
\right\rfloor .
\end{equation}
Then we get the following result: \vskip.3cm

\begin{corollary}
\label{th:non-auto} Let $m$ be a natural number and $f:\mathbb{R}%
\times\mathbb{R}^{k}\rightarrow\mathbb{R}^{k}$ be a continuous function
satisfying the assumptions \textrm{($A_{1}$)---($A_{3}$)}, \textrm{($A_{5}%
$)---($A_{6}$)} and $\sigma_{-}(A)$ consists a single simple eigenvalue $\mu$
satisfying \eqref{eq:simple-ex}. Suppose $m=2^{\varepsilon_{0}}p_{1}%
^{\varepsilon_{1}}p_{2}^{\varepsilon_{2}}\dots p_{s}^{\varepsilon_{s}}$, where
$p_{l}>2$ are distinct prime numbers, $\varepsilon_{0}\geq0$, $\varepsilon
_{l}\geq1$ for $l=1,2,\dots,s$. Then

\begin{itemize}
\item[(i)] ~~the system \eqref{eq:sys-non-aut} admits $G$-orbit of non-zero
${2\pi m}$-periodic solutions with symmetries $(D_{m})$;

\item[(ii)] ~~ if for some $l=1,2,\dots,s$, the number $\frac{p_{l}-1}{2}$ is
odd,  the system \eqref{eq:sys-non-aut} admits $G$-orbit of ${2\pi m}$-periodic
solutions with symmetries $(D_{m_{l}}^{z})$ or $(D_{m}^{z})$ (here
$m_{l}=\frac{m}{p_{l}}$);

\item[(iii)] ~~ if $\varepsilon_{0}>0$, the system \eqref{eq:sys-non-aut}
admits $G$-orbit of ${2\pi m}$-periodic solutions with symmetries exactly
$(D_{m}^{d})$, $(D_{m}^{\hat d})$;

\item[(iv)] ~~ if $\varepsilon_{0}>1$, the system \eqref{eq:sys-non-aut}
admits $G$-orbit of ${2\pi m}$-periodic solutions with symmetries exactly
$(D_{\frac m2}^{d})$, $(\widetilde D_{\frac m2}^{d})$;
\end{itemize}
\end{corollary}

\begin{proof}
Notice that we have
\[
\beta_{0}(\mu)=1,\;\;\beta_{\mathfrak{s}}(\mu)=0,\;\;\beta_{i}(\mu)=1\;\text{
for }\;0<i<\frac{m}{2},
\]
and in the case $m$ is even
\[
\beta_{\mathfrak{s}}(\mu+1)=\beta_{\mathfrak{s}}(\mu+2)=1.
\]
Consequently,

\begin{itemize}
\item Since $\rho_{0}=\beta_{0}=1$, thus there exist a non-zero $2\pi
m$-periodic solution to the system \eqref{eq:sys-non-aut} with symmetries
exactly $(D_{m})$;

\item Since $\rho_{\mathfrak{s}}=\beta_{\mathfrak{s}}=0$, we cannot conclude
the existence of a $2\pi m$-periodic solution to the system
\eqref{eq:sys-non-aut} with symmetries exactly $(D_{m}^{z})$;

\item Since $\rho_{m_{l}}=\beta_{m_{l}}\cdot\frac{p_{l}-1}{2}=\frac{p_{l}%
-1}{2}$, $l=1,2,\dots,s$, if $\frac{p_{l}-1}{2}$ is odd there exists a $2\pi
m$-periodic solution to the system \eqref{eq:sys-non-aut} with symmetries
either $(D_{m_{l}}^{z})$ or $(D_{m}^{z})$ (notice that $(D_{m_{l}}^{z})$ is
not a maximal orbit type);

\item In the case $m$ is even, i.e. $\varepsilon_{0}>0$, since $\rho
_{\mathfrak{s}+1}=\rho_{\mathfrak{s}+2}=1$, it follows that the system
\eqref{eq:sys-non-aut} has an orbit of $2\pi m$-periodic solutions with
symmetries either $(D_{m}^{d})$ or $(D_{m}^{\hat{d}})$;

\item In addition, if $\varepsilon_{0}>1$, notice that $\rho_{\frac{m}{4}%
}=\beta_{\frac{m}{4}}=1$, thus the system \eqref{eq:sys-non-aut} admits
$G$-orbit of ${2\pi m}$-periodic solutions with symmetries $(D_{\frac{m}{2}%
}^{d})$, $(\widetilde{D}_{\frac{m}{2}}^{d})$;
\end{itemize}
\end{proof}

\vskip.3cm In order to illustrate that the other maximal orbit types in 
 in $\mathbb{E}\setminus\{0\}$ can also appear as symmetries of $2\pi m$-periodic subharmonic solutions to  \eqref{eq:sys-non-aut},
 we assume that $m=2^{n}p_{1}^{\varepsilon_{1}%
}p_{2}^{\varepsilon_{2}}\dots p_{s}^{\varepsilon_{s}}$,  $n>1$, and the operator $A$
has one simple negative eigenvalue $\mu$ satisfying
\begin{equation}
-\frac{(\mathfrak{p}+1)^{2}}{m^{2}}<\mu<-\frac{\mathfrak{p}^{2}}{m^{2}}%
,\quad\mathfrak{p}:=\left\lfloor \frac{m}{2^{n}}\right\rfloor
.\label{eq:simple-ex2}%
\end{equation}
Then we have: \vskip.3cm

\begin{corollary}
Let $m$ be a natural number and $f:\mathbb{R}\times\mathbb{R}^{k}%
\rightarrow\mathbb{R}^{k}$ be a continuous function satisfying the assumptions
\textrm{($A_{1}$)---($A_{3}$)}, \textrm{($A_{5}$)---($A_{6}$)} and $\sigma
_{-}(A)$ consists a single simple eigenvalue $\mu$ satisfying
\eqref{eq:simple-ex2}. Suppose $m=2^{n}p_{1}^{\varepsilon_{1}}p_{2}%
^{\varepsilon_{2}}\dots p_{s}^{\varepsilon_{s}}$, where $n>2$, $p_{l}>2$ are
distinct prime numbers, $\varepsilon_{0}\geq0$, $\varepsilon_{l}\geq1$ for
$l=1,2,\dots,s$. Then

\begin{itemize}
\item[(i)] ~~The system \eqref{eq:sys-non-aut} admits $G$-orbit of non-zero
${2\pi m}$-periodic solutions with symmetries $(D_{m}),$

\item[(ii)] ~~ The system \eqref{eq:sys-non-aut} admits $G$-orbit of ${2\pi
m}$-periodic solutions with symmetries exactly $(D_{\frac{m}{2^{n-1}}}^{d})$,
$(\widetilde{D}_{\frac{m}{2^{n-1}}}^{d})$;
\end{itemize}
\end{corollary}

\vskip.3cm

\subsection{Examples of Symmetric Systems}

\label{ex:existence} 

In this section we assume that $k=3$, $f:\mathbb{R}\times \mathbb{R}%
^{3}\rightarrow \mathbb{R}^{3}$ satisfies assumptions ($A_{1}$)---($A_{5}$)
with $\Gamma =D_{3}$ and $\ $ 
\begin{equation}
A=\frac{1}{4}\left[ 
\begin{array}{ccc}
-4 & -2 & -2 \\ 
-2 & -4 & -2 \\ 
-2 & -2 & -4%
\end{array}%
\right] .  \label{eq:A-mat}
\end{equation}%
Then we have 
\[
\sigma (A)=\left\{\mu _{0}=-2,\mu _{1}=-\frac{1}{2}\right\}.
\]%
\vskip.3cm

\paragraph{\textbf{Case $m=3$: }}

In this case $G=D_{3}\times D_{3}\times {\mathbb{Z}}_{2}$ and we have 
\[
\sigma (\mathscr A)=\left\{ \lambda _{j,l}:=1+\frac{9(\mu _{l}-1)}{j^{2}+9}%
:l=0,1,\;j=0,1,2,\dots \right\} 
\]%
Then 
\begin{align*}
\sigma _{-}(\mathscr A)=& \Big\{\lambda _{0,0}=-2,\;\lambda _{1,0}=-\frac{17%
}{10},\;\lambda _{2,0}=-\frac{14}{13},\;\lambda _{3,0}=-\frac{1}{2}%
,\;\lambda _{4,0}=-\frac{2}{25}, \\
& \lambda _{0,1}=-\frac{1}{2},\;\lambda _{1,1}=-\frac{7}{20},\;\lambda
_{2,1}=\allowbreak -\frac{1}{26}\Big\},
\end{align*}
and each of the eigenvalues $\lambda_{l,i}$ are isotypicaly simple, i.e. the
eigenspaces $E(\lambda_{l,i})$ are irreducible $G$-representations. More
precisely, we have
\begin{gather*}
E(\lambda_{0,0})\simeq\mathcal{V}_{0,0}^{-},\quad E(\lambda_{3,0}%
)\simeq\mathcal{V}_{0,0}^{-}\oplus\mathcal{V}_{2,0}^{-},\quad E(\lambda
_{1,0})\simeq E(\lambda_{2,0})\simeq E(\lambda_{4,0})\simeq\mathcal{V}%
_{1,0}^{-},\\
E(\lambda_{0,1})\simeq\mathcal{V}_{1,0}^{-},\quad E(\lambda_{1,1})\simeq
E(\lambda_{2,1})\simeq\mathcal{V}_{1,1}^{-}.
\end{gather*}
Therefore, we obtain
\begin{align*}
G\text{\textrm{-deg}}(\mathscr F,\Omega) &  =(G)-G\text{\textrm{-deg}%
}(\mathscr A,B_{1}(0))\\
&  =(G)-(\deg_{\mathcal{V}_{0,0}^{-}})^{2}\circ\deg_{\mathcal{V}_{2,0}^{-}%
}\circ(\deg_{\mathcal{V}_{1,0}^{-}})^{3}\circ\deg_{\mathcal{V}_{0,1}^{-}}%
\circ(\deg_{\mathcal{V}_{1,1}^{-}})^{2}\\
&  =(G)-\deg_{\mathcal{V}_{0,1}^{-}}\circ\deg_{\mathcal{V}_{1,0}^{-}}\circ
\deg_{\mathcal{V}_{2,0}^{-}}%
\end{align*}

\noindent{\small \textbf{GAP Code:} We use GAP package \texttt{EquiDeg} to
compute $G\text{\textrm{-deg}}(\mathscr F,\Omega)$ for the groups
$G:=D_{3}\times D_{3}\times{\mathbb{Z}}_{2}$. The GAP code f is given below. }

{\small \begin{lstlisting}[language=GAP, frame=single]
LoadPackage( "EquiDeg" );
gr1 := pDihedralGroup( 3 );
gr2 := SymmetricGroup( 2 );
# create the product of D_3 and Z_2
gr3:= DirectProduct( gr1, gr2 );
# create group G
G := DirectProduct( gr1, gr3 );
# create and name CCSs of gr1 and gr3
ccs_g1:= ConjugacyClassesSubgroups( gr1 );
ccs_g1_names := [ "Z1", "D1", "Z3", "D3" ];
ccs_gr3:=ConjugacyClassesSubgroups( gr3 );
ccs_gr3_names:=["Z1","Z1p","D1","D1z", "Z3",
"D1p","Z3p","D3","D3z","D3p"];
SetCCSsAbbrv(gr1, ccs_g1_names);
SetCCSsAbbrv(gr3, ccs_gr3_names);
ccs := ConjugacyClassesSubgroups( G );
cc := ConjugacyClasses( G );
# create characters of irreducible G-representations
irr := Irr( G );
# compute the corresponding to irr[k[] basic degree degk
\end{lstlisting}
For a subgroup $K\le D_{n}$, we use the symbol $K^{p}:=K\times{\mathbb{Z}}%
_{2}$, but in the code we simply write \texttt{Kp}. } By using the list of
conjugacy classes \texttt{cc}, one can easily recognize the irreducible
$G$-representations $\mathcal{V} _{i,l}^{\pm}:=\mathcal{V} _{i}^{\pm}%
\otimes\mathcal{U} _{l}$. For example, the character

\begin{center}\scriptsize
\begin{tabular}
[c]{|c|ccccccccc|}\hline
& $(\pm\bm1)$ & $(\pm\mathbf{\kappa}_{2})$ & $(\pm\bm\gamma_{2})$ &
$(\pm\mathbf{\kappa}_{1})$ & $(\pm\bm\kappa_{1} \mathbf{\kappa}_{2})$ &
$(\pm\mathbf{\kappa}_{1}\mathbf{\gamma}_{2})$ & $(\pm\mathbf{\gamma}_{1})$ &
$(\pm\mathbf{\gamma}_{1}\mathbf{\kappa}_{2})$ & $(\pm\mathbf{\gamma}%
_{1}\mathbf{\gamma}_{2})$\\\hline
\texttt{Irr(G)[9]} & $\pm2$ & $\mp2$ & $\pm2$ & $0$ & $0$ & $0$ & $\mp1$ &
$\pm1$ & $\mp1$\\\hline
\end{tabular}

\end{center}

\noi is the character of the $G$-irreducible representation $\mathcal{V}^{-}_{2,1}%
$. To be more precise, we have the following correspondence of the characters.

\begin{center}%
\begin{tabular}
[c]{|c|c|c|}\hline
~~~\texttt{Irr(G)[1]}$=\chi_{\mathcal{V} ^{+}_{0,0}}$~~~ &
~~~\texttt{Irr(G)[7]}$=\chi_{\mathcal{V} ^{+}_{2,0}}$~~~ &
~~~\texttt{Irr(G)[13]}$=\chi_{\mathcal{V} ^{-}_{1,2}}$~~~\\
~~~\texttt{Irr(G)[2]}$=\chi_{\mathcal{V} ^{-}_{2,2}}$~~~ &
~~~\texttt{Irr(G)[8]}$=\chi_{\mathcal{V} ^{+}_{0,2}}$~~~ &
~~~\texttt{Irr(G)[14]}$=\chi_{\mathcal{V} ^{-}_{1,0}}$~~~\\
~~~\texttt{Irr(G)[3]}$=\chi_{\mathcal{V} ^{-}_{2,0}}$~~~ &
~~~\texttt{Irr(G)[9]}$=\chi_{\mathcal{V} ^{-}_{2,1}}$~~~ &
~~~\texttt{Irr(G)[15]}$=\chi_{\mathcal{V} ^{+}_{1,2}}$~~~\\
~~~\texttt{Irr(G)[4]}$=\chi_{\mathcal{V} ^{-}_{0,2}}$~~~ &
~~~\texttt{Irr(G)[10]}$=\chi_{\mathcal{V} ^{-}_{0,1}}$~~~ &
~~~\texttt{Irr(G)[16]}$=\chi_{\mathcal{V} ^{+}_{1,0}}$~~~\\
~~~\texttt{Irr(G)[5]}$=\chi_{\mathcal{V} ^{-}_{0,0}}$~~~ &
~~~\texttt{Irr(G)[11]}$=\chi_{\mathcal{V} ^{+}_{2,1}}$~~~ &
~~~\texttt{Irr(G)[17]}$=\chi_{\mathcal{V} ^{+}_{1,1}}$~~~\\
~~~\texttt{Irr(G)[6]}$=\chi_{\mathcal{V} ^{+}_{2,2}}$~~~ &
~~~\texttt{Irr(G)[12]}$=\chi_{\mathcal{V} ^{+}_{0,1}}$~~~ &
~~~\texttt{Irr(G)[18]}$=\chi_{\mathcal{V} ^{-}_{1,1}}$~~~\\\hline
\end{tabular}

\end{center}

To conclude the computations in GAP (we continue to use the package
\texttt{EquiDeg}). \noindent
{\small \begin{lstlisting}[language=GAP, frame=single]
# unit element in A(G)
u := -BasicDegree( Irr( G )[1] );
# basic degrees
deg01 := BasicDegree( Irr( G )[10] );
deg10 := BasicDegree( Irr( G )[14] );
deg20 := BasicDegree( Irr( G )[3] );
deg := u-deg01*deg10*deg20;
\end{lstlisting}}

The list of conjugacy classes of $G=D_{3}\times D_{3}\times{\mathbb{Z}}_{2}$
generated by GAP is $\{(H_{k}):1\le k\le69\}$, where $(G)=(H_{69})$. As a
result we obtain
\begin{align}
G\text{\textrm{-deg}}(\mathscr F,\Omega) & =-(H_{1})+(H_{4})+(H_{6}%
)+(H_{8})+(H_{11})-2(H_{18})-(H_{22})-(H_{27})\nonumber\\
& +(H_{35})+(H_{43})+(H_{44})+(H_{45})-(H_{51})+(H_{63})- (H_{67}%
),\label{eq:D3-max-OT}
\end{align}
where the coefficient of $G\text{\textrm{-deg}}(\mathscr F,\Omega)$ can be
easily described using amalgamated notation, for example \noindent
{\small \begin{lstlisting}
gap> Print( AmalgamationSymbol( ccs[45] ) );
\end{lstlisting}} In this way we get the following description of some orbit
types represented in $G\text{\textrm{-deg}}(\mathscr F,\Omega)$:
\begin{alignat*}{3}
H_{4} & =-(D_{1}\times_{{\mathbb{Z}}_{2}}D_{1}^{z}),\quad H_{6} &
=(D_{1}\times{\mathbb{Z}}_{1}),\\
H_{8} & =({\mathbb{Z}}_{1}\times D_{1}),\quad H_{18} & =(D_{1}\times D_{1}),\\
H_{22} & =({\mathbb{Z}}_{1}\times D_{3}), \quad H_{21} & =(D_{1}%
\times{\mathbb{Z}}_{3}),\\
H_{43} & =(D_{3}\times D_{1}), \quad H_{44} & =(D_{1}\times D_{3}),\\
H_{45} & =(D_{1}\times_{{\mathbb{Z}}_{2}}^{D_{3}}D_{3}^{p}), \quad H_{51} &
=(D_{3}\times D_{1}^{z}),\\
H_{63} & =(D_{3}\times D_{3}^{z}), \quad H_{67} & =(D_{3}\times D_{3}).
\end{alignat*}

Next we find the maximal orbit types in the representation $\mathbb{E}$:

\noindent{\small \begin{lstlisting}[language=GAP, frame=single]
# characters appearing in E
chi :=  Irr(G)[3]+  Irr(G)[5]+ Irr(G)[9]+ Irr(G)[10]
+ Irr(G)[14]+ Irr(G)[18];
# find orbit types in E
orbtyps := ShallowCopy( OrbitTypes( chi ) );
Remove( orbtyps );
# find maximal orbit types in H-0
max_orbtyps := MaximalElements( orbtyps );
Print( List( max_orbtyps, IdCCS ) );
\end{lstlisting}}

The maximal orbit types in $\mathbb{E}\setminus\{0\}$ are
\begin{gather*}
(H_{45})=(D_{1}\times_{{\mathbb{Z}}_{2}}^{D_{3}}D_{3}^{p}),\; \;\;(H_{52}%
)=(D_{1}\times_{{\mathbb{Z}}_{2}}^{D_{3}^{z}}D_{3}^{p}) ,\\
(H_{63})=(D_{3}\times D_{3}^{z}),\;\;\; (H_{67})=(D_{3}\times D_{3}).
\end{gather*}
Consequently, we obtain the following \vskip.3cm

\begin{theorem}
Let $m=3$ and $\Gamma=D_{3}=\langle(1,2,3),(2,3)\rangle$ (acting on
$\mathbb{R}^{3}$ by permuting the coordinates). Assume that $f:\mathbb{R}%
\times\mathbb{R}^{3}\rightarrow\mathbb{R}^{3}$ and $A$ (given by
\eqref{eq:A-mat}) satisfy the conditions ($A_{1}$)--($A_{6}$). Then there exist

\begin{itemize}
\item[(i)] at least one $(D_{1}\times_{{\mathbb{Z}}_{2}}^{D_{3}}D_{3}^{p}%
)$-orbit of (i.e. at least 6 different) $6\pi$-periodic solutions to the
system \eqref{eq:sys-non-aut},

\item[(ii)] at least one $(D_{3}\times D_{3}^{z})$-orbit of (i.e. at least 2
different) $6\pi$-periodic solutions to the system \eqref{eq:sys-non-aut},

\item[(iii)] at least one $(D_{3}\times D_{3})$-orbit of (i.e. at least 2
different) $6\pi$-periodic solutions to the system \eqref{eq:sys-non-aut},
\end{itemize}

Therefore, the system \eqref{eq:sys-non-aut} admits at least 10 different
$6\pi$-periodic solutions.
\end{theorem}

\vskip.3cm

\paragraph{\textbf{Case $m=4$:}}

In this case $G=D_{4}\times D_{3}\times {\mathbb{Z}}_{2}$, we have 
\[
\sigma (\mathscr A)=\left\{ \lambda _{j,l}:=1+\frac{16(\mu _{l}-1)}{j^{2}+16}%
:l=0,1,\;j=0,1,2,\dots \right\} 
\]%
Then 
\begin{align*}
\sigma _{-}(\mathscr A)=& \Big\{\lambda _{0,0}=-2,\;\lambda _{1,0}=-\frac{31%
}{17},\;\lambda _{2,0}=-\frac{7}{5},\;\lambda _{3,0}=-\frac{23}{25}%
,\;\lambda _{4,0}=-\frac{1}{2}, \\
& \lambda _{5,0}=-\frac{7}{41},\;\lambda _{0,1}=-\frac{1}{2},\;\lambda
_{1,1}=-\frac{7}{17},\;\lambda _{2,1}=-\frac{1}{5}\Big\},
\end{align*}
each of the eigenvalues $\lambda_{l,i}$ are isotypicly simple, i.e. the
eigenspaces $E(\lambda_{l,i})$ are irreducible $G$-representations. More
precisely, we have
\begin{gather*}
E(\lambda_{0,0})\simeq\mathcal{V}_{0,0}^{-},\quad E(\lambda_{4,0}%
)\simeq\mathcal{V}_{0,0}^{-}\oplus\mathcal{V}_{2,0}^{-},\quad E(\lambda
_{1,0})\simeq E(\lambda_{3,0})\simeq E(\lambda_{5,0})\simeq\mathcal{V}%
_{0,1}^{-},\\
E(\lambda_{2,0})\simeq\mathcal{V}_{3,0}^{-}\oplus\mathcal{V}_{4,0}%
^{-},\;E(\lambda_{0,1})\simeq\mathcal{V}_{0,1}^{-},\quad E(\lambda
_{1,1})\simeq\mathcal{V}_{1,1}^{-},\;E(\lambda_{2,1})\simeq\mathcal{V}%
_{3,1}^{-}\oplus\mathcal{V}_{4,1}^{-}%
\end{gather*}
Therefore, we obtain
\begin{align*}
G\text{\textrm{-deg}}(\mathscr F,\Omega) &  =(G)-G\text{\textrm{-deg}%
}(\mathscr A,B_{1}(0))\\
&  =(G)-(\deg_{\mathcal{V}_{0,0}^{-}})^{2}\circ(\deg_{\mathcal{V}_{2,0}^{-}%
})\circ(\deg_{\mathcal{V}_{0,1}^{-}})^{3}\circ(\deg_{\mathcal{V}_{3,0}^{-}%
})\circ(\deg_{\mathcal{V}_{4,0}^{-}})\\
&  \hskip.7cm\circ\deg_{\mathcal{V}_{0,2}^{-}}\circ\deg_{\mathcal{V}_{1,0}%
^{-}}\circ\deg_{\mathcal{V}_{1,1}^{-}}\circ\deg_{\mathcal{V}_{3,1}^{-}}%
\circ(\deg_{\mathcal{V}_{4,1}^{-}})\\
&  =(G)-(\deg_{\mathcal{V}_{2,0}^{-}})\circ(\deg_{\mathcal{V}_{0,1}^{-}}%
)\circ(\deg_{\mathcal{V}_{3,0}^{-}})\circ(\deg_{\mathcal{V}_{4,0}^{-}}%
)\circ\deg_{\mathcal{V}_{0,2}^{-}}\\
&  \hskip.7cm\circ\deg_{\mathcal{V}_{1,0}^{-}}\circ\deg_{\mathcal{V}_{1,1}%
^{-}}\circ\deg_{\mathcal{V}_{3,1}^{-}}\circ(\deg_{\mathcal{V}_{4,1}^{-}})
\end{align*}
\noindent{\small \textbf{GAP Code:} $G:=D_{4}\times D_{3}\times{\mathbb{Z}%
}_{2}$. }

{\small \begin{lstlisting}[language=GAP, frame=single]
LoadPackage( "EquiDeg" );
gr1 := pDihedralGroup( 3 );
gr2 := SymmetricGroup( 2 );
# create the product of D_3 and Z_2
gr3:= DirectProduct( gr1, gr2 );
# create group G
gr4 := pDihedralGroup( 4 );
G := DirectProduct( gr4, gr3 );
# create and name CCSs of gr4 and gr3
ccs_g1:= ConjugacyClassesSubgroups( gr1 );
ccs_g4_names := [ "Z1", "Z2", "D1",  "tD1", "D2",
"Z4", "tD2", "D4" ];
ccs_gr3:=ConjugacyClassesSubgroups( gr3 );
ccs_gr3_names:=["Z1","Z1p","D1","D1z", "Z3",
"D1p","Z3p","D3","D3z","D3p"];
SetCCSsAbbrv(gr4, ccs_g4_names);
SetCCSsAbbrv(gr3, ccs_gr3_names);
ccs := ConjugacyClassesSubgroups( G );
cc := ConjugacyClasses( G );
# create characters of irreducible G-representations
irr := Irr( G );
# compute the corresponding to irr[k[] basic degree degk
\end{lstlisting}
For a subgroup $K\le D_{n}$, we use the symbol $K^{p}:=K\times{\mathbb{Z}}%
_{2}$ and in the code we simply write \texttt{Kp}. } The group $G=D_{4}\times
D_{3}\times{\mathbb{Z}}_{2}$ has 236 conjugacy classes of subgroups. The
conjugacy classes are denoted $(H_{k})$, $k=1,2,\dots,236$, and are according
to the same order as it is generated by GAP. The group G has 30 irreducible
representations which can be easily identified in GAP.

\begin{center}%
\begin{tabular}
[c]{|c|c|c|}\hline
~~~\texttt{Irr(G)[2]}$=\chi_{\mathcal{V} ^{-}_{2,2}}$~~~ &
~~~\texttt{Irr(G)[7]}$=\chi_{\mathcal{V} ^{-}_{3,0}}$~~~ &
~~~\texttt{Irr(G)[19]}$=\chi_{\mathcal{V} ^{-}_{4,1}}$~~~\\
~~~\texttt{Irr(G)[3]}$=\chi_{\mathcal{V} ^{-}_{3,2}}$~~~ &
~~~\texttt{Irr(G)[8]}$=\chi_{\mathcal{V} ^{-}_{4,0}}$~~~ &
~~~\texttt{Irr(G)[20]}$=\chi_{\mathcal{V} ^{-}_{0,1}}$~~~\\
~~~\texttt{Irr(G)[4]}$=\chi_{\mathcal{V} ^{-}_{4,2}}$~~~ &
~~~\texttt{Irr(G)[9]}$=\chi_{\mathcal{V} ^{-}_{0,0}}$~~~ &
~~~\texttt{Irr(G)[25]}$=\chi_{\mathcal{V} ^{-}_{1,2}}$~~~\\
~~~\texttt{Irr(G)[5]}$=\chi_{\mathcal{V} ^{-}_{0,2}}$~~~ &
~~~\texttt{Irr(G)[17]}$=\chi_{\mathcal{V} ^{-}_{2,1}}$~~~ &
~~~\texttt{Irr(G)[26]}$=\chi_{\mathcal{V} ^{-}_{1,0}}$~~~\\
~~~\texttt{Irr(G)[6]}$=\chi_{\mathcal{V} ^{-}_{2,0}}$~~~ &
~~~\texttt{Irr(G)[18]}$=\chi_{\mathcal{V} ^{-}_{3,1}}$~~~ &
~~~\texttt{Irr(G)[30]}$=\chi_{\mathcal{V} ^{-}_{1,1}}$~~~\\\hline
\end{tabular}

\end{center}

Therefore, we are set up to compute the degree $G\text{\textrm{-deg}%
}(\mathscr F,\Omega)$: \noindent
{\small \begin{lstlisting}[language=GAP, frame=single]
# unit element in A(G)
u := -BasicDegree( Irr( G )[1] );
# basic degrees
deg20 := BasicDegree( Irr( G )[6] );
deg01 := BasicDegree( Irr( G )[20] );
deg30 := BasicDegree( Irr( G )[7] );
deg40 := BasicDegree( Irr( G )[8] );
deg02 := BasicDegree( Irr( G )[5] );
deg10 := BasicDegree( Irr( G )[26] );
deg11 := BasicDegree( Irr( G )[30] );
deg31 := BasicDegree( Irr( G )[18] );
deg41 := BasicDegree( Irr( G )[19] );
deg := u-deg20*deg01*deg30* deg40* deg02* deg10
* deg11* deg31* deg41;
\end{lstlisting}}

We also use the list of all irreducible $G$-representations generated by GAP.
Using this list, the corresponding basic $G$-degrees are easily computed by
the GAP program, so the exact value of $G\text{\textrm{-deg}}%
(\mathscr F,\Omega)$ is given by {\small
\begin{align*}
G\text{\textrm{-deg}}(\mathscr F,\Omega) & =-3(H_{1})+(H_{2})+(H_{3})
+(H_{4})-(H_{6})+2(H_{8})+(H_{9})\\
&  +(H_{10}) +2(H_{11})-(H_{12}) + (H_{13})+(H_{14}) -(H_{17}) +(H_{18})
-(H_{19})\\
& -(H_{20}) -(H_{23}) -(H_{25})-(H_{26}) +(H_{27}) -(H_{28}) -(H_{32})
+(H_{33})\\
&  -(H_{34})+(H_{35}) -(H_{36})-(H_{39}) +(H_{40}) -(H_{43})- (H_{55}) +
2(H_{57})\\
&  -(H_{60}) +(H_{62}) +(H_{68}) +(H_{71}) +(H_{72}) -(H_{73}) -(H_{76})
+(H_{86})\\
& +(H_{87}) -(H_{88}) -(H_{90}) -(H_{92}) -(H_{95}) +(H_{99}) +(H_{100})
+(H_{102})\\
& +(H_{103})-(H_{104}) +(H_{105}) +(H_{109}) --(H_{116}) +(H_{117}) +(H_{119})
-(H_{120})\\
& -(H_{122}) +(H_{130}) -(H_{135})- (H_{138}) -(H_{144}) +(H_{148}) +(H_{151})
+(H_{167})\\
& +(H_{169}) +(H_{170})-m (H_{173}) +(H_{174}) + (H_{175})+(H_{177}%
)+(H_{179})-(H_{192})\\
& -(H_{203}) -(H_{208}) - (H_{212}) -(H_{214}) +(H_{223}) +(H_{225})
+(H_{229}) +(H_{233}).
\end{align*}
}

\noindent{\small \begin{lstlisting}[language=GAP, frame=single]
# characters appearing in E
chi :=  Irr(G)[2]+Irr(G)[3]+ Irr(G)[4]+ Irr(G)[5]
+ Irr(G)[6]+ Irr(G)[7]+Irr(G)[8]+Irr(G)[9]
+ Irr(G)[17]+ Irr(G)[18]+ Irr(G)[19]+ Irr(G)[20]
+ Irr(G)[25]+ Irr(G)[26]+ Irr(G)[30];
# find orbit types in E
orbtyps := ShallowCopy( OrbitTypes( chi ) );
Remove( orbtyps );
# find maximal orbit types in H-0
max_orbtyps := MaximalElements( orbtyps );
Print( List( max_orbtyps, IdCCS ) );
\end{lstlisting}}

Since the $G$-isotypic components in $\mathbb{E}$ are easily identified, the
GAP program also allows a quick computation of all maximal orbit types in
$\mathbb{E}\setminus\{0\}$, namely
\begin{gather*}
(H_{177}),\;\;(H_{178}),\;\;(H_{179}),\;\;(H_{180}),\;\;(H_{223}%
),\;\;(H_{224}),\;\;(H_{225}),\\
(H_{228}),\;\;(H_{229}),\;\;(H_{232}),\;\;(H_{233}).
\end{gather*}
One can notice that $G\text{\textrm{-deg}}(\mathscr F,\Omega)$ has non-zero
coefficients for the following maximal orbit types:
\begin{alignat*}{3}
(H_{177}) &  =(D_{2}^{D_{1}}\times_{{\mathbb{Z}}_{2}}^{D_{3}}D_{3}^{p}%
),\quad(H_{179}) &  =(\widetilde{D}_{2}^{\widetilde{D}_{1}}\times
_{{\mathbb{Z}}_{2}}^{D_{3}}D_{3}^{p})\\
(H_{223}) &  =(D_{4}\times D_{3}^{z}),\quad(H_{225}) &  =(D_{4}^{D_{2}}%
\times_{{\mathbb{Z}}_{2}}^{D_{3}}D_{3}^{p}),\quad\\
(H_{229}) &  =(D_{4}^{D_{2}}\times_{{\mathbb{Z}}_{2}}^{D_{3}}D_{3}%
^{p}),(H_{233}) &  =(D_{4}^{{\mathbb{Z}}_{4}}\times_{{\mathbb{Z}}_{2}}^{D_{3}%
}D_{3}^{p})
\end{alignat*}

Consequently, we obtain the following result. \vskip.3cm

\begin{theorem}
Let $m=4$, $k=3$ and $\Gamma=D_{3}=\langle(1,2,3),(2,3)\rangle$ (acting on
$\mathbb{R}^{3}$ by permuting the coordinates). Assume that $f:\mathbb{R}%
\times\mathbb{R}^{3}\rightarrow\mathbb{R}^{3}$ and $A$ (given by
\eqref{eq:A-mat}) satisfy the conditions ($A_{1}$)---($A_{6}$). Then there exist

\begin{itemize}
\item[(i)] at least one $(D_{2}^{ D_{1}} \times^{D_{3}}_{{\mathbb{Z}}_{2}%
}D_{3}^{p})$-orbit of (i.e. at least 4 different) $8\pi$-periodic solutions to
the system \eqref{eq:sys-non-aut},

\item[(ii)] at least one $(\widetilde D_{2}^{ \widetilde D_{1}} \times^{D_{3}%
}_{{\mathbb{Z}}_{2}}D_{3}^{p})$-orbit of (i.e. at least 4 different) $8\pi
$-periodic solutions to the system \eqref{eq:sys-non-aut},

\item[(iii)] at least one $(D_{4}\times D_{3}^{z})$-orbit of (i.e. at least 2
different) $8\pi$-periodic solutions to the system \eqref{eq:sys-non-aut},

\item[(iv)] at least one $(D_{4}^{D_{2}} \times^{D_{3}}_{{\mathbb{Z}}_{2}%
}D_{3}^{p})$-orbit of (i.e. at least 2 different) $8\pi$-periodic solutions to
the system \eqref{eq:sys-non-aut},

\item[(v)] at least one $( D_{4}^{ D_{2}} \times^{D_{3}}_{{\mathbb{Z}}_{2}%
}D_{3}^{p}) $-orbit of (i.e. at least 2 different) $8\pi$-periodic solutions
to the system \eqref{eq:sys-non-aut},

\item[(vi)] at least one $(D_{4}^{{\mathbb{Z}}_{4}}\times^{D_{3}}%
_{{\mathbb{Z}}_{2}}D_{3}^{p})$-orbit of (i.e. at least 2 different) $8\pi
$-periodic solutions to the system \eqref{eq:sys-non-aut}.
\end{itemize}

Therefore, the system \eqref{eq:sys-non-aut} admits at least 16 different
$8\pi$-periodic solutions.
\end{theorem}

\vskip.3cm

\section{Bifurcation in Reversible Non-Autonomous Second Order Differential
Equaitons}

\label{sec:app2:3} Consider the following parametrized system with $2\pi
$-periodic coefficients
\begin{equation}
\ddot{u}(t)=\big(-\alpha\id+A\big)u(t)+f(t,u(t)),\;\;u(t)\in\mathbb{R}%
^{k},\label{eq:sys-non-aut-bif}%
\end{equation}
where $A$ is a non-singular $k\times k$-matrix and $f:\mathbb{R}%
\times\mathbb{R}^{k}\rightarrow\mathbb{R}^{k}$ is a continuous function
satisfying the conditions ($A_{1}$)---($A_{3}$) and :

\begin{itemize}
\item[($B_{4}$)] ~~ $\displaystyle \lim_{x\to0} \frac{f(t,x)}{|x|}=0$
uniformly with respect to $t$.
\end{itemize}

We are interested in studying the bifurcation of the \textit{subharmonic}
${2\pi m}$-periodic solutions (for some integer $m$) to
\eqref{eq:sys-non-aut-bif} from the trivial solution $(\alpha,0)$, i.e. the
solutions
\begin{equation}
u(t)=u(t+{2\pi m}),\;\;\dot{u}(t)=\dot{u}(t+{2\pi m}).\label{eq:pm-per-bif}%
\end{equation}
We also consider a subgroup $\Gamma\leq S_{k}$ which acts on $V:=\mathbb{R}%
^{k}$ by permuting the coordinates of vectors $x=(x_{1},x_{2},\dots,x_{k}%
)^{T}$ in $\mathbb{R}^{k}$ given by
\begin{equation}
\sigma x=\sigma(x_{1},x_{2},\dots,x_{k})^{T}:=(x_{\sigma(1)},x_{\sigma
(2)},\dots,x_{\sigma(k)})^{T}.\label{eq:G-act-V}%
\end{equation}
Clearly, the space $V:=\mathbb{R}^{k}$ equipped with this $\Gamma$-action is
an orthogonal $\Gamma$-representation. We also introduce the following conditions

\begin{itemize}
\item[($B_{5}$)] ~~ For all $t\in\mathbb{R}$, $x\in V$ and $\sigma\in\Gamma$,
we have $f(t,\sigma x)=\sigma f(t,x)$ and $A\sigma x=\sigma Ax$;
\end{itemize}

The condition ($B_{5}$) implies that the system \eqref{eq:sys-non-aut-bif} is
$\Gamma$-symmetric. The bifurcation problem \eqref{eq:sys-non-aut-bif} with
the boundary conditions \eqref{eq:pm-per-bif} can be expressed as the
following equation
\begin{equation}
\mathscr F(\alpha,u)=0,\quad(\alpha,u)\in\mathbb{R}\oplus\mathbb{E}%
,\label{eq:bif-eq}%
\end{equation}
where
\begin{equation}
\mathscr F(\alpha,u):=u-L^{-1}\Big(N_{A+f}\big(\mathfrak{j}(u)\big)-(\alpha
+1)\mathfrak{j}(u)\Big),\quad\alpha\in\mathbb{R},\;u\in\mathbb{E}%
.\label{eq:F-bif}%
\end{equation}
Put $G:=\Gamma\times D_{m}\times{\mathbb{Z}}_{2}$. Notice that under the
assumptions ($A_{1}$)--($A_{3}$) and ($B_{5}$), the map $\mathscr F$ is
$G$-equivariant completely continuous field such that $\mathscr F(\alpha,0)=0$
for all $\alpha\in\mathbb{R}$. Moreover, the assumption ($B_{4}$) implies that
$\mathscr F$ is differentiable at $(\alpha,0)$ with respect to $u\in
\mathbb{E}$, and
\[
\mathscr A(\alpha):=D_{u}\mathscr F(\alpha,0)=\id-L^{-1}\Big(N_{A}%
\circ\mathfrak{j}-(\alpha+1)\mathfrak{j}\Big):\mathbb{E}\rightarrow\mathbb{E}.
\]
The necessary condition for the point $(\alpha_{o},0)$ to be a bifurcation
point for \eqref{eq:bif-eq} is that $\mathscr A(\alpha_{o}):\mathbb{E}%
\rightarrow\mathbb{E}$ is {\bf not an isomorphism}, i.e. $0\in\sigma
(\mathscr A(\alpha_{o})$.
\vs

The point $\alpha_{o}$ is called a \textit{critical point} for
\eqref{eq:bif-eq} and the set of all such critical points $\alpha_{o}$ is
denoted $\Lambda$. One can easily compute the spectrum of the operator
$\mathscr A(\alpha_{o})$:
\[
\sigma(\mathscr A(\alpha_{o})):=\left\{  1+\frac{m^{2}(\mu-\alpha_{o}%
-1)}{j^{2}+m^{2}}:j=0,1,2,\dots,\;\mu\in\sigma(A)\right\}  ,
\]
which implies that
\[
\Lambda=\left\{  \alpha_{j,\mu}:=\frac{j^{2}+m^{2}\mu}{m^{2}}:j=0,1,2,\dots
,\;\mu\in\sigma(A)\right\}  .
\]

\vs
\subsection{Bifurcation in System \eqref{eq:sys-non-aut-bif} without
Symmetries}

Assume that $\Gamma=\{e\}$, i.e. $G:=D_{m}\times{\mathbb{Z}}_{2}$, and that
$\alpha_{j,\mu}\not =\alpha_{j^{\prime},\mu^{\prime}}$ for $(j,\mu
)\not =(j^{\prime},\mu^{\prime})$. Let us put all the elements of $\Lambda$ in
increasing order, i.e. $\dots<\alpha_{j_{k},\mu_{k}}<\alpha_{j_{k+1},\mu
_{k+1}}<\dots$. Then for every $\alpha_{j_{o},\mu_{o}}\in\Lambda$ we have
\begin{equation}
\label{eq:G-inv}\omega_{G}(\alpha_{j_{o},\mu_{o}})=\prod_{\alpha_{j_{k}%
,\mu_{k}}<\alpha_{j_{o},\mu_{o}}} (\deg_{\mathcal{W} _{i(j_{k})}^{-}}%
)^{m(\mu_{k})}\Big((G)-(\deg_{\mathcal{W} _{i(j_{o})}^{-}})^{m(\mu_{o})}\Big).
\end{equation}
\vskip.3cm

\begin{theorem}
\label{th:bif-non-eq} Suppose $A:V\rightarrow V$ and $f:\mathbb{R}\times
V\rightarrow V$ satisfies the assumptions ($A_{1}$)---($A_{3}$) and ($B_{4}$)
and let $\Lambda$ be the critical set for \eqref{eq:bif-eq}. Assume that for
all $\alpha_{j,\mu}$, $\alpha_{j^{\prime},\mu^{\prime}}\in\Lambda$ we have
$\alpha_{j,\mu}\not =\alpha_{j^{\prime},\mu^{\prime}}$ if $(j,\mu
)\not =(j^{\prime},\mu^{\prime})$. Then for every $\alpha_{j_{o},\mu_{o}}%
\in\Lambda$ such that $m(\mu_{o})$ is odd we have $\omega(\alpha_{j_{o}%
,\mu_{o}})\not =0$, i.e. the point $(\alpha_{j_{o},\mu_{o}},0)$ is a
bifurcation point of non-trivial ${2\pi m}$-periodic solutions for \eqref{eq:sys-non-aut-bif}.
\end{theorem}

\begin{proof}
Notice that under the assumption that $m( \mu_{o})$ is odd, we have
\[
(G)-(\deg_{\mathcal{W} _{i(j_{o})}^{-}})^{m(\mu_{o})}=(G)-(\deg_{\mathcal{W}
_{i(j_{o})}^{-}})\not =0.
\]
Since the product $\displaystyle \prod_{\alpha_{j_{k},\mu_{k}}<\alpha
_{j_{o},\mu_{o}}} (\deg_{\mathcal{W} _{i(j_{k})}^{-}})^{m(\mu_{k})}$ is an
invertible element in $A(G)$, it follows from formula \eqref{eq:G-inv} that
$\omega_{G}(\alpha_{j_{o},\mu_{o}})$ is non-zero. Therefore, by Theorem
\ref{th:comp-map-bif}, the point $( \alpha_{j_{o},\mu_{o}},0)$ is a
bifurcation point of non-trivial ${2\pi m}$-periodic solutions for \eqref{eq:sys-non-aut-bif}.
\end{proof}

\vskip.3cm Consequently, we obtain the following: \vskip.3cm

\begin{theorem}
\label{th:bif-non-eq2} Suppose $A:V\rightarrow V$ and $f:\mathbb{R}\times
V\rightarrow V$ satisfies the assumptions ($A_{1}$)---($A_{3}$) and ($B_{4}%
$). Assume that $k>0$ is odd and $\sigma(A)$ consists of exactly $k
$ different eigenvalues such that for all $(j,\mu)\not =(j^{\prime}%
,\mu^{\prime})$ the critical points $\alpha_{j,\mu}$, $\alpha_{j^{\prime}%
,\mu^{\prime}}\in\Lambda$ are also different. Suppose $m=2^{\ve_0}p_{1}^{\varepsilon
_{1}}p_{2}^{\varepsilon_{2}}\dots p_{s}^{\varepsilon_{s%
}}$, where $\ve_0\ge 0$, $\varepsilon_{l}>0$ and $p_{l}$ are the prime numbers such that
$2\leq p_{1}<p_{2}<\dots<p_{s}$. For $l=1,2,\dots,s$ put
$m_{l}:=\frac{m}{p_{l}}$. Then

\begin{itemize}
 \itemindent=2pt\labelsep=3pt\labelwidth5pt\itemsep=1pt

\item[(a)] if $p_{l}>2$ and $\rho_{l}$ is odd, then the system
\eqref{eq:bif-eq} admits an unbounded branch of ${2\pi m}$-periodic solutions
with symmetries at least $(D_{m_{l}}^{z})$, $m_{l}:=\frac m{p_{l}}$, and

\item[(b)] if $p_{1}=2$, $\varepsilon_{1}=1$, and $\rho_{1}$ is odd, then the
system \eqref{eq:bif-eq} admits an unbounded branch of ${2\pi m}$-periodic
solutions with symmetries at least $(D_{m}^{d})$, $(D_{m}^{\hat d})$, and

\item[(c)] if $p_{1}=2$ and $\varepsilon_{1}>1$, then and $\rho_{l}$ is odd,
then the system \eqref{eq:bif-eq} admits an unbounded branch of ${2\pi m}%
$-periodic solutions with symmetries at least $(D_{\frac{m}{2}}^{d})$,
$(\widetilde{D}_{\frac{m}{2}}^{d})$.
\end{itemize}
\end{theorem}

\vskip.3cm

\subsection{Bifurcation in System \eqref{eq:sys-non-aut-bif} with Additional
Symmetries $\Gamma$}

In this section we assume that $k=3$, $f:\mathbb{R}\times\mathbb{R}%
^{3}\rightarrow\mathbb{R}^{3}$ satisfies assumptions ($A_{1}$)---($A_{3}$) and ($B_{4}$)---($B_{5}$)
with $\Gamma=D_{3}$ and 
 
\begin{equation}
A=\frac{1}{4}\left[ 
\begin{array}{ccc}
-4 & -2 & -2 \\ 
-2 & -4 & -2 \\ 
-2 & -2 & -4%
\end{array}%
\right] .  \label{eq:A-mat}
\end{equation}%
Then we have 
\[
\sigma (A)=\{\mu _{0}=-2,\mu _{1}=-\frac{1}{2}\}.
\]

\paragraph{\textbf{Case $m=3$:}}

In this case $G=D_{3}\times D_{3}\times {\mathbb{Z}}_{2}$ and we have 
\[
\sigma (\mathscr A(\alpha ))=\left\{ \lambda _{j,l}:=1+\frac{9(\mu
_{l}-\alpha -1)}{j^{2}+9}:l=0,1,\;j=0,1,2,\dots \right\} 
\]%
Then%
\begin{align*}
\Lambda =& \Big\{\alpha _{0,0}=-2,\;\alpha _{0,1}=-1/2,\;\alpha _{1,0}=-%
\frac{17}{9},\;\alpha _{1,1}=-\frac{7}{18},\dots , \\
& \alpha _{j,0}=\frac{j^{2}-18}{9},\;\alpha _{j,1}=\frac{j^{2}-9/2}{9}%
,\;\dots \Big\},
\end{align*}
and
each of the critical values $\alpha_{j,i}$ is isotypicly simple, i.e. the
eigenspace $E(\lambda_{j,i})$ is an irreducible $G$-representation
$\mathcal{V}_{i(j),l}^{-}$. Maximal orbit types in $\mathbb{E}\setminus\{0\}$
are listed in \eqref{eq:D3-max-OT}. Using the same GAP code as in Example (for
$D_{3}$) in subsection \ref{ex:existence}, we can compute the exact bifurcation
invariants. Indeed, we have the following critical points from $\Lambda$
\[
\alpha_{0,0}<\alpha_{1,0}<\alpha_{2,0}<\alpha_{0,1}<\alpha
_{1,1}<\alpha_{2,1}<\alpha_{3,0}<\alpha_{3,1}<\alpha_{4,0}<\alpha_{4,1}<\alpha_{5,1}<\dots
\]
Thus
\begin{align*}
\omega(\alpha_{0,0}) &  =(G)-\deg_{\mathcal{V}_{0,0}^{-}}=(H_{67})\\
\omega(\alpha_{1,0}) &  =\deg_{\mathcal{V}_{0,0}^{-}}\circ((G)-\deg
_{\mathcal{V}_{1,0}^{-}})=-(H_{43})+(H_{51})\\
\omega(\alpha_{2,0}) &  =\deg_{\mathcal{V}_{0,0}^{-}}\circ\deg_{\mathcal{V}%
_{1,0}^{-}}\circ((G)-\deg_{\mathcal{V}_{1,0}^{-}})=(H_{43})-(H_{51})\\
\omega(\alpha_{0,1}) &  =\deg_{\mathcal{V}_{0,0}^{-}}\circ((G)-\deg
_{\mathcal{V}_{0,0}^{-}}\circ\deg_{\mathcal{V}_{2,0}^{-}}=(H_{63})+(H_{67}))\\
\omega(\alpha_{1,1}) &  =\deg_{\mathcal{V}_{2,0}^{-}}\circ((G)-\deg
_{\mathcal{V}_{1,1}^{-}})=(H_{1})-2(H_{3})+(H_{4})-(H_{5})-(H_{8})+(H_{13})\\
&  +(H_{14})-(H_{16})+(H_{18})-(H_{28})+(H_{31})\\
\omega(\alpha_{4,0}) &  =\deg_{\mathcal{V}_{2,0}^{-}}\circ\deg_{\mathcal{V}%
_{1,1}^{-}}\circ((G)-\deg_{\mathcal{V}_{1,0}^{-}})=(H_{3})-(H_{4})-(H_{7})\\
&  +(H_{8})+2(H_{16})-2(H_{18})+(H_{43})-(H_{51})\\
\omega(\alpha_{2,1}) &  =\deg_{\mathcal{V}_{2,0}^{-}}\circ\deg_{\mathcal{V}%
_{1,1}^{-}}\circ\deg_{\mathcal{V}_{1,0}^{-}}\circ((G)-\deg_{\mathcal{V}%
_{1,1}^{-}})=-(H_{1})+(H_{3})+(H_{5})\\
&  +(H_{7})-(H_{13})-(H_{14})-(H_{16})+(H_{18})+(H_{28})-(H_{31})
\end{align*}

\vskip.3cm

\vskip.3cm \appendix

\section{APPENDIX: $G$-Equivariant Brouwer Degree}

\paragraph{\textbf{Equivariant Notation:}}

For a subgroup $H$ of $G$, i.e. $H\leq G$, we denote the normalizer of $H$ in
$G$ by $N(H)$ and the Weyl group of $H$ by $W(H)=N(H)/H$. The symbol $(H) $
stands for the conjugacy class of $H$ in $G$. We put $\Phi(G):=\{(H): H\le
G\}$, i.e. $\Phi(G)$ is the set of conjugacy classes of subgroups in $G$.
$\Phi(G)$ has a natural partial order defined by $(H)\leq(K)$ iff $\exists
g\in G\;\;gHg^{-1}\leq K$. For two subgroups $L$, $H\le G$, we denote  by $n(L,H)$ the
number of different subgroups $H^{\prime}$ conjugate to $H$ such that $L\le
H^{\prime}$.
\vs

For a $G$-space $X$ and $x\in X$, we denote the {\it isotropy group} of $x
$ by $G_{x} :=\{g\in G:gx=x\}$, the {\it orbit} of $x$ by $G(x) :=\{gx:g\in
G\}$, and we call $(G_{x})$ the {\it orbit type} of $x\in X$. For a
subgroup $H\leq G$ the subspace $X^{H} :=\{x\in X:G_{x}\geq H\}$ is called the
{\it $H$-fixed-point subspace} of $X$. Clearly, $W(H)$ acts on $X^{H}$.
\vs

Consider a complete list of all irreducible $G$-representations $\mathcal{V}%
_{i}$, $i=0,$ $1,$ $\ldots$, $r $. Such list for concrete group $G$ can be
established by using GAP. Let $V$ be a finite-dimensional $G$- representation
and (without loss of generality) we may assume that $V$ is an orthogonal
representation. Then $V$ decomposes into a direct sum
\begin{equation}
V=V_{0}\oplus V_{1}\oplus\dots\oplus V_{r},\label{eq:Giso}%
\end{equation}
where each component $V_{i}$ is \textit{modeled} on the irreducible
$G$-representation $\mathcal{V}_{i}$, $i=0,1,2,\dots,r$, that is $V_{i}$
contains all the irreducible subrepresentations of $V$ equivalent to
$\mathcal{V}_{i}$. The decomposition \eqref{eq:Giso} is called $G$%
{\it-isotypic decomposition of } $V$. Denote the $\mathbb{R}$-algebra
(resp. group) of all $G$-equivariant linear (resp. invertible) operators on
$V$ by $\text{L}^{G}(V)$ (resp. $\text{GL}^{G}(V)$) . Clearly, the isotypic
decomposition \eqref{eq:Giso} induces the following direct sum decomposition
of $\text{GL}^{G}(V):$
\begin{equation}
\text{GL}^{G}(V)=\bigoplus_{i=0}^{r}\text{GL}^{G}(V_{i}),\label{eq:GLG-decomp}%
\end{equation}
where for every isotypic component $V_{i}$
\[
\text{GL}^{G}(V_{i})\simeq\text{GL}(m_{i},\mathbb{F}),\quad m_{i}%
=\text{\textrm{dim\,}}V_{i}/\text{\textrm{dim\,}}\mathcal{V}_{i}%
\]
and depending on the type of the irreducible representation $\mathcal{V}_{i},
$ $\mathbb{F}$ ($=\mathbb{R}$, $\mathbb{C}$ or $\mathbb{H}$) is a
finite-dimensional division algebra $\text{L}^{G}(\mathcal{V}_{i})$.
\vs

\subsection{Burnside Ring}

We assume that $G$ is a finite group. Denote the free abelian group generated
by $(H)\in\Phi(G)$ by $A(G):={\mathbb{Z}}[\Phi(G)]$, i.e., an element $a\in
A(G)$ can be written as a sum
\begin{align*}
a=n_{1}(H_{1})+\dots+n_{m}(H_{m}),
\end{align*}
where $n_{i}\in\mathbb{Z} $ and $(H_{i})\in\Phi(G)$. There is a natural
multiplication operation $\circ:A(G)\times A(G)\to A(G)$ which is defined on
generators $(H)$, $(K)\in\Phi(G)$ by
\begin{equation}
\label{eq:Burnside-mul}(H)\cdot(K)=\sum_{(L)\in\Phi(G)}m_{L}\,(L),
\end{equation}
where the integer $m_{L}$ represents the number of $(L)$-orbits contained in
the space $G/H\times G/K$. The numbers $m_{L}$ can be easily computed from the
following \textit{recurrence formula}
\begin{equation}
\label{eq:rec-mult}m_{L}=\frac{n(L,H)|W(H)|n(L,K)|W(K)|-\sum_{(\widetilde
L)>(L)} m_{\widetilde L}\, n(L,\widetilde L)\,|W(\widetilde L)|}{|W(L)|}.
\end{equation}
Together with the multiplication `$\circ$', $A(G)$ becomes a ring with the
unity $(G)$, which is called the {\it Burnside ring} of $G$. For more details see \cite{AED}

\subsection{Axioms of Brouwer $G$-Equivariant Degree}

\label{subsec:G-degree}
Consider an orthogonal $G$-representation $V$, a continuous $G$-map $f:V\to V
$, and an open bounded $G$-invariant set $\Omega\subset V$ such that for all
$x\in\partial\Omega$, we have $f(x)\neq0$. Then $f$ is called \textit{$\Omega
$-admissible} and $(f,\Omega)$ is called a \textit{$G$-admissible pair} (in
$V$). The set of all possible $G$-pairs will be denoted by $\mathcal{M}^{G}$. \vskip.3cm

The following result (cf \cite{AED}) can be considered as an axiomatic definition of the
{\it $G$-equivariant Brouwer degree}:

\begin{theorem}
\label{thm:GpropDeg} There exists a unique map $G\mbox{\rm -}\deg:\mathcal{M}
^{G}\to A(G)$, which assigns to every admissible $G$-pair $(f,\Omega)$ an
element $G\text{\textrm{-deg}}(f,\Omega)\in A(G)$
\begin{equation}
\label{eq:G-deg0}G\mbox{\rm -}\deg(f,\Omega)=\sum_{(H)}{n_{H}(H)}= n_{H_{1}%
}(H_{1})+\dots+n_{H_{m}}(H_{m}),
\end{equation}
satisfying the following properties:

\begin{itemize}
\item \textbf{(Existence)} If $G\mbox{\rm -}\deg(f,\Omega)\ne0$, i.e.,
$n_{H_{i}}\neq0$ for some $i$ in \eqref{eq:G-deg0}, then there exists
$x\in\Omega$ such that $f(x)=0$ and $(G_{x})\geq(H_{i})$.

\item \textbf{(Additivity)} Let $\Omega_{1}$ and $\Omega_{2}$ be two disjoint
open $G$-invariant subsets of $\Omega$ such that $f^{-1}(0)\cap\Omega
\subset\Omega_{1}\cup\Omega_{2}$. Then
\begin{align*}
G\mbox{\rm -}\deg(f,\Omega)=G\mbox{\rm -}\deg(f,\Omega_{1})+G\mbox{\rm -}\deg
(f,\Omega_{2}).
\end{align*}

\item \textbf{(Homotopy)} If $h:[0,1]\times V\to V$ is an $\Omega$-admissible
$G$-homotopy, then
\begin{align*}
G\mbox{\rm -}\deg(h_{t},\Omega)=\mathrm{constant}.
\end{align*}

\item \textbf{(Normalization)} Let $\Omega$ be a $G$-invariant open bounded
neighborhood of $0$ in $V$. Then
\begin{align*}
G\mbox{\rm -}\deg(\id,\Omega)=(G).
\end{align*}

\item \textbf{(Multiplicativity)} For any $(f_{1},\Omega_{1}),(f_{2}%
,\Omega_{2})\in\mathcal{M} ^{G}$,
\begin{align*}
G\mbox{\rm -}\deg(f_{1}\times f_{2},\Omega_{1}\times\Omega_{2})= G\mbox{\rm
-}\deg(f_{1},\Omega_{1})\circ G\mbox{\rm -}\deg(f_{2},\Omega_{2}),
\end{align*}
where the multiplication `$\circ$' is taken in the Burnside ring $A(G )$.

\item \textbf{(Suspension)} If $W$ is an orthogonal $G$-representation and
$\mathscr B$ is an open bounded invariant neighborhood of $0\in W$, then
\begin{align*}
G\mbox{\rm -}\deg(f\times\id_{W},\Omega\times\mathscr B)=G\mbox{\rm -}\deg
(f,\Omega).
\end{align*}

\item \textbf{(Recurrence Formula)} For an admissible $G$-pair $(f,\Omega)$,
the $G$-degree \eqref{eq:G-deg0} can be computed using the following
recurrence formula
\begin{equation}
\label{eq:RF-0}n_{H}=\frac{\deg(f^{H},\Omega^{H})- \sum_{(K)>(H)}{n_{K}\,
n(H,K)\, \left|  W(K)\right|  }}{\left|  W(H)\right|  },
\end{equation}
where $\left|  X\right|  $ stands for the number of elements in the set $X$
and $\deg(f^{H},\Omega^{H})$ is the Brouwer degree of the map $f^{H}
:=f|_{V^{H}}$ on the set $\Omega^{H}\subset V^{H}$.
\end{itemize}
\end{theorem}

The $G\text{\textrm{-deg}}(f,\Omega)$ is called the {\it $G$-equivariant
Brouwer degree} (or simply {\it $G$-degree}) of $f$ in $\Omega$.

\subsection{Computation of $\Gamma$-Equivariant Degree}

Put $B(V):=\left\{  x\in V:\left|  x\right|  <1\right\}  $. For each
irreducible $G$-representation $\mathcal{V} _{i}$, $i=0,1,2,\dots$, we define
\begin{align*}
\deg_{\mathcal{V}_{i}}:=G\mbox{\rm -}\deg(-\id,B(\mathcal{V} _{i})),
\end{align*}
and will call $\deg_{\mathcal{V} _{i}}$ the \emph{basic degree}.
\vs

Consider a $G$-equivariant linear isomorphism $T:V\to V$ and assume that $V$
has a $G$-isotypic decomposition \eqref{eq:Giso}. Then by the Multiplicativity
property, we have
\begin{equation}
\label{eq:prod-prop}G\mbox{\rm -}\deg(T,B(V))=\prod_{i=0}^{r}G\mbox{\rm -}\deg
(T_{i},B(V_{i}))= \prod_{i=0}^{r}\prod_{\mu\in\sigma_{-}(T)} \left(
\deg_{\mathcal{V} _{i}}\right)  ^{m_{i}(\mu)}%
\end{equation}
where $T_{i}=T|_{V_{i}}$ and $\sigma_{-}(T):=\left\{  \mu\in\sigma(T):\mu<0\right\}$ denotes the real negative
spectrum of $T$.
 \vskip.3cm

Notice that the basic degrees can be effectively computed from
\eqref{eq:RF-0}, i.e. formula
\begin{align*}
\deg_{\mathcal{V} _{i}}=\sum_{(H)}n_{H}(H),
\end{align*}
where
\begin{equation}
\label{eq:bdeg-nL}n_{H}=\frac{(-1)^{\dim\mathcal{V} _{i}^{H}}- \sum
_{H<K}{n_{K}\, n(H,K)\, \left|  W(K)\right|  }}{\left|  W(H)\right|  }.
\end{equation}

\vs
\section{Local and Global Bifurcation Problems}

Assume that $a<b$, $V$ is an orthogonal $G$-representation and $\Omega\subset
V$ is an open bounded $G$-invariant subset. Let $f:\mathbb{R}\oplus V \to V$
be a continuous $G$-equivariant map such that $(f_{a},\Omega)$, $(f_{b}%
,\Omega)\in\mathcal{M}^{G}( V)$, where $f_{t}(x):=f(t,x)$, $t\in\mathbb{R}$,
$x\in V$. Then a continuous $G$-invariant function $\varphi:\mathbb{R}\oplus
V\to\mathbb{R}$ will be called \textit{$\Omega$-complementing function} for
$f_{t}$ at $t=a$, $b$, if
\begin{equation}
\label{eq:vp-comp}%
\begin{cases}
\varphi(t,x)<0 & \;\text{ if }\; t=a,\, b, \; x\in\Omega\\
\varphi(t,x)>0 & \;\text{ if }\; t=(a,b), \; x\in\partial\Omega.
\end{cases}
\end{equation}
In such a case we define the map $F_{\varphi}:\mathbb{R}\oplus V\to
\mathbb{R}\oplus V$ by
\begin{equation}
\label{eq:comp-map}F_{\varphi}(t,x)=(\varphi(t,x),f(t,x)), \quad
t\in\mathbb{R}, \; x\in V.
\end{equation}
The following result is well-known in non-equivariant case.  \vskip.3cm

\begin{theorem}
\label{th:comp-map-bif} Suppose that $f:\mathbb{R}\oplus V\to V$ is a
$G$-equivariant map such that $(f_{a},\Omega)$, $(f_{b},\Omega)\in
\mathcal{M}^{G}( V) $, and $\varphi:\mathbb{R}\oplus V\to V$ is an $\Omega
$-complementing function for $f_{t}$ at $t=a$, $b$. Then $(F_{\varphi
},(a,b)\times\Omega)\in\mathcal{M}^{G}$, the $G$-equivariant degree
$G\text{\textrm{-deg}}(F_{\varphi},(a,b)\times\Omega)$ doesn't depend on the
choice of the $\Omega$-complementing function $\varphi$ and we have
\begin{equation}
\label{eq:bif-G-deg}\gdeg(F_{\varphi},(a,b)\times
\Omega)=\gdeg(f_{a},\Omega)-\gdeg(f_{b}%
,\Omega).
\end{equation}

\end{theorem}

\vskip.3cm Let $f:\mathbb{R}\times V\to V$ be a continuous $G$-equivariant map
such that for all $t\in\mathbb{R}$, $f(t,0)=0$. We are interested in solutions
of the equation
\begin{equation}
\label{eq:bif-G-para}f(t,x)=0.
\end{equation}
Clearly any pair $(t,0)$ satisfies \eqref{eq:bif-G-para}, thus we will call
them a \textit{trivial solutions} to \eqref{eq:bif-G-para}. All other
solutions to \eqref{eq:bif-G-para} will be called \textit{nontrivial}. We
denote the set of all nontrivial solutions to \eqref{eq:bif-G-para} by
$\mathscr S$, i.e.
\[
\mathscr S:=\{(t,x)\in\mathbb{R}\times\mathbb{R}^{n}\oplus V: f(t,x)=0\;\text{
and } x\not =0 \}.
\]

\vs
\begin{definition}\rm
Let $\mathcal{C}\subset\mathscr S$ and $\mathscr U\subset\mathbb{R}\times V\to
V$ be a $G$-invariant open subset. The set $\mathcal{C}\subset\mathscr S$ is
called a \textit{branch} of nontrivial solutions to \eqref{eq:bif-G-para} in
$\mathscr U$ if $\overline{\mathcal{C}}$ is a connected component of
$\overline{\mathscr S} \cap\overline{\mathscr U}$. Moreover, we say that the
branch $\mathcal{C}$ bifurcates from a trivial solution $(t_{o},0)$ if
$(t_{o},0)\in\overline{\mathcal{C}}$.
\end{definition}

\vskip.3cm


\begin{theorem}
\label{th:continuum} Let $V$ be an orthogonal $G$-representation and
$\Omega\subset V$ be a $G$-invariant bounded open set. Assume that
$f:[0,1]\times V\to V$ is a continuous $G$-equivariant map such that for every
$t\in[0,1]$, $(f_{t},\Omega)$ is an admissible $G$-pair and $\deg_{G}%
(f_{0},\Omega)\not =0$. Then there exists a compact connected set
$K_{o}\subset f^{-1}(0)\cap[0,1]\times\Omega$ such that
\[
K_{o}\cap( \{0\}\times\Omega)\not =\emptyset\not =K_{o}\cap( \{1\}\times
\Omega).
\]

\end{theorem}

\vskip.3cm Let us discuss the global bifurcation problem for
\eqref{eq:bif-G-para}. Under the assumption that the derivative $D_{x}f(t,0)$
exists for all $t\in\mathbb{R}$ and the map $t\mapsto D_{x}f(t,0)$ is
continuous, one can easily show that if $(t_{o},0)$ is a bifurcation point for
\eqref{eq:bif-G-para}, then $D_{x}f(t_{o},0)$ is not an isomorphism. Denote
the set of all $t\in\mathbb{R}$ such that $(t,0)$ is a bifurcation point of
\eqref{eq:bif-G-para} by $\mathscr B$ and put
\begin{equation}
\label{eq:critB}\Lambda:=\{t\in\mathbb{R}: \det D_{x} f(t,0)=0\}.
\end{equation}
$\Lambda$ is called the set of \textit{critical points} for
\eqref{eq:bif-G-para} and we clearly have $\mathscr B\subset\Lambda$. \vskip.3cm

The following result is called \textit{Rabinowitz's Alternative}: \vskip.3cm

\begin{theorem}
\label{th:Rab} Suppose that $f:\mathbb{R}\oplus V\to V$ is a continuous
$G$-equivariant map such that $f(t,0)=0$ for all $t\in\mathbb{R}$ and
$D_{x}f(t,0)$ exists and is continuous with respect to $t\in\mathbb{R}$. We
also assume that the set of critical points $\Lambda$ for
\eqref{eq:bif-G-para} (given by \eqref{eq:critB}) is discrete, and consider an
open bounded $G$-invariant set $\mathcal{U}\subset\mathbb{R}\oplus V$ such
that $(t_{o},0)\in\mathcal{U}$ for some $t_{o}\in\Lambda$. For a connected
component $\mathscr C$ of the set $\overline{\mathcal{U}}\cap\overline
{\mathscr S}$ such that $(t_{o},0)\in\mathscr C$ we have the following alternative.

\begin{itemize}
 \itemindent=2pt\labelsep=3pt\labelwidth5pt\itemsep=1pt

\item[(i)] ~~either $\mathscr C\cap\partial U\not = \emptyset$,

\item[(ii)] ~~or $\mathscr C\cap(\Lambda\times\{0\})=\{ (t_{1},0),(t_{2}%
,0),\dots,(t_{n},0)\}$ for some $n\in{\mathbb{N}}$ (here $t_{j}\not =t_{k}$
for $j\not = k$) and
\begin{equation}
\label{eq:Rab}\sum_{k=1}^{n} \omega_{G}(t_{k})=0.
\end{equation}

\end{itemize}
\end{theorem}

\end{document}